\documentclass[11pt]{amsart}
\usepackage{amscd,amssymb}
\usepackage{amsthm,amsmath,amssymb}
\usepackage[matrix,arrow]{xy}

\theoremstyle{definition}
\newtheorem{example}[equation]{Example}
\newtheorem{definition}[equation]{Definition}
\newtheorem{theorem}[equation]{Theorem}
\newtheorem{lemma}[equation]{Lemma}
\newtheorem{corollary}[equation]{Corollary}

\newtheorem*{question*}{Question}
\newtheorem{problem}[equation]{Problem}
\newtheorem*{problem*}{Problem}

\theoremstyle{remark}
\newtheorem{remark}[equation]{Remark}

\makeatletter\@addtoreset{equation}{section} \makeatother

\author{Ivan Cheltsov and Andrew Wilson}

\title{Del Pezzo surfaces with many symmetries}

\address{\begin{tabbing}
\hspace*{28 em}\=\kill School of Mathematics, University of Edinburgh, Edinburgh EH9 3JZ, UK\\
\end{tabbing}}

\thanks{The~first author would like to thank Institut des Hautes Etudes Scientifiques  for hospitality.}%

\begin{document}

\begin{abstract}
We classify smooth del Pezzo surfaces whose $\alpha$-invariant~of
Tian is bigger than $1$.
\end{abstract}

\maketitle

We assume that all varieties are projective, normal, and defined
over $\mathbb{C}$.

\section{Introduction}
\label{section:intro}

Let $X$ be a~smooth Fano variety, and let $G$ be a finite subgroup
in $\mathrm{Aut}(X)$.~Put
$$
\mathrm{lct}_{n}\big(X, G\big)=\mathrm{sup}\left\{\lambda\in\mathbb{Q}\ \left|%
\aligned
&\text{the~log pair}\ \Bigg(X, \frac{\lambda}{n} D\Bigg)\ \text{is log canonical}\\
&\text{for any $G$-invariant divisor}\ D\in \big|-nK_{X}\big|\\
\endaligned\right.\right\}\in\mathbb{Q}\cup\big\{+\infty\big\}%
$$
for every $n\in\mathbb{N}$. Then $\mathrm{lct}_{n}(X)\ne+\infty$
$\iff$ $|-nK_{X}|$ contains a~$G$-invariant divisor. Put
$$
\mathrm{lct}\big(X,G\big)=\mathrm{inf}\Big\{\mathrm{lct}_{n}\big(X, G\big)\ \Big|\ n\in\mathbb{N}\Big\}\in\mathbb{R},%
$$
and put $\mathrm{lct}(X)=\mathrm{lct}(X,G)$ in the case when $G$
is a trivial group.

\begin{example}[{\cite[Theorem~1.7]{Ch07b}}]
\label{example:GAFA} Suppose that $\mathrm{dim}(X)=2$. Then
$$
\mathrm{lct}\big(X\big)=\left\{%
\aligned
&1\ \mathrm{when}\ K_{X}^{2}=1\ \mathrm{and}\ |-K_{X}|\ \mathrm{has\ no\ cuspidal\ curves},\\%
&5/6\ \mathrm{when}\ K_{X}^{2}=1\ \mathrm{and}\ |-K_{X}|\ \mathrm{has\ a\ cuspidal\ curve},\\%
&5/6\ \mathrm{when}\ K_{X}^{2}=2\ \mathrm{and}\ |-K_{X}|\ \mathrm{has\ no\ tacnodal\ curves},\\%
&3/4\ \mathrm{when}\ K_{X}^{2}=2\ \mathrm{and}\ |-K_{X}|\ \mathrm{has\ a\ tacnodal\ curve},\\%
&3/4\ \mathrm{when}\ X\ \mathrm{is\ a\ cubic\ surface\ in}\ \mathbb{P}^{3}\ \mathrm{without\ Eckardt\ points},\\%
&2/3\ \mathrm{when}\ K_{X}^{2}=4\ \mathrm{or}\ X\ \mathrm{is\ a\ cubic\ surface\ in}\ \mathbb{P}^{3}\ \mathrm{with\ an\ Eckardt\ point},\\%
&1/2\ \mathrm{when}\ X\cong\mathbb{P}^{1}\times\mathbb{P}^{1}\ \mathrm{or}\ K_{X}^{2}\in\big\{5,6\big\},\\%
&1/3\ \mathrm{in\ the\ remaining\ cases}.\\%
\endaligned\right.%
$$
\end{example}

The number $\mathrm{lct}(X,G)$ plays an important role in K\"ahler
Geometry, since
$$
\mathrm{lct}\big(X,G\big)=\alpha_{G}\big(X\big)
$$
by \cite[Theorem~A.3]{ChSh08c},~where $\alpha_{G}(X)$ is
the~$\alpha$-invariant introduced in \cite{Ti87}.

\begin{theorem}[{\cite{Ti87}}]
\label{theorem:KE} The  variety $X$ admits a~$G$-invariant
K\"ahler--Einstein metric~if
$$
\mathrm{lct}\big(X,G\big)>\frac{\mathrm{dim}\big(X\big)}{\mathrm{dim}\big(X\big)+1}.
$$
\end{theorem}

The problem of the~existence of K\"ahler-Einstein metrics on
smooth del Pezzo surfaces~is~solved.

\begin{theorem}[{\cite{Tian90}}]
\label{thm:E-KE-metric} If $\mathrm{dim}(X)=2$, then the~following
conditions are equivalent:
\begin{itemize}
\item the~surface $X$ admits a~K\"ahler-Einstein metric,%
\item the~surface $X$ is not the~blow up of $\mathbb{P}^2$ in one or two points.%
\end{itemize}
\end{theorem}

Let $g_{0}=g_{i\overline{j}}$ be a~$G$-invariant K\"ahler metric
on the~variety $X$ with a~K\"ahler form
$$
\omega_{0}=\frac{\sqrt{-1}}{2\pi}\sum g_{i\overline{j}}dz_i\wedge d\overline{z}_j\in \mathrm{c}_1\big(X\big),%
$$
and let $\omega_{1},\omega_{2},\ldots,\omega_{m}$ be K\"ahler
forms of some $G$-invariant metrics on $X$ such that
\begin{equation}
\label{equation:iteration} \left\{\aligned
&\mathrm{Ric}\big(\omega_{m}\big)=\omega_{m-1},\\
&\cdots\\
&\mathrm{Ric}\big(\omega_{2}\big)=\omega_{1},\\
&\mathrm{Ric}\big(\omega_{1}\big)=\omega_{0},\\
\endaligned
\right.
\end{equation}
and $\omega_{i}\in\mathrm{c}_1(X)$ for every $i$. By \cite{Yau78},
a solution to $(\ref{equation:iteration})$ always exist.

\begin{theorem}[{\cite[Theorem~3.3]{Rub08}}]
\label{theorem:Yanir} Suppose that $\mathrm{lct}(X,G)>1$. Then in
$C^{\infty}(X)$-topology
$$
\lim_{m\to+\infty}\omega_{m}=\omega_{KE},
$$
where $\omega_{KE}$ is a~K\"ahler form of a~$G$-invariant
K\"ahler--Einstein metric on the~variety $X$.
\end{theorem}

Smooth Fano varieties that satisfy all hypotheses of
Theorem~\ref{theorem:Yanir} do exist.

\begin{example}
\label{example:Shokurov} If $X\cong\mathbb{P}^{1}$, then
$\mathrm{lct}(\mathbb{P}^{1},G)>1$ $\iff$ either
$G\cong\mathbb{A}_{4}$ or $G\cong\mathbb{S}_{4}$ or
$G\cong\mathbb{A}_{5}$.
\end{example}

\begin{theorem}[{\cite[Lemma~2.30]{ChSh08c}}]
\label{theorem:lct-product} Let $X_{1}$ and $X_{2}$ be smooth Fano
varieties. Then
$$
\mathrm{lct}\Big(X_{1}\times X_{2}, G_{1}\times G_{2}\Big)=\mathrm{min}\Big(\mathrm{lct}\big(X_1,G_1\big),\mathrm{lct}\big(X_2,G_2\big)\Big),%
$$
where $G_{1}$ and $G_{2}$
are~finite~subgroups~in~$\mathrm{Aut}(X_{1})$ and
$\mathrm{Aut}(X_{2})$ respectively.
\end{theorem}

\begin{corollary}
\label{corollary:binary-groups} Let $G_{1}$ and $G_{2}$ be finite
subgroups in $\mathrm{Aut}(\mathbb{P}^1)$. Then
$$
\mathrm{lct}\Big(\mathbb{P}^1\times\mathbb{P}^1, G_{1}\times G_{2}\Big)>1\iff G_{1}\in\{\mathbb{A}_{4},\mathbb{S}_{4},\mathbb{A}_{5}\}\ni G_{2}.%
$$
\end{corollary}

The purpose of this paper is to consider the~following two
problems.

\begin{problem}
\label{problem:I} Describe all smooth del Pezzo surfaces that
satisfy all hypotheses of Theorem~\ref{theorem:Yanir}.
\end{problem}

\begin{problem}
\label{problem:II} For a~smooth del Pezzo surface $X$ that satisfy
all hypotheses of Theorem~\ref{theorem:Yanir}, describe all finite
subgroups of the~group $\mathrm{Aut}(X)$ that satisfy all
hypotheses of Theorem~\ref{theorem:Yanir}.
\end{problem}

There exists a~partial solution to Problem~\ref{problem:I} (cf.
Corollary~\ref{corollary:binary-groups}).

\begin{example}[{\cite{MarPr99}, \cite{Ch07b},  \cite{ChSh09a}}]
\label{example:partial-answer} If $\mathrm{dim}(X)=2$ and
$\mathrm{Aut}(X)$ is finite, then
\begin{itemize}
\item $\mathrm{lct}(X,\mathrm{Aut}(X))=2$ if $X$ is the Clebsch
cubic surface in $\mathbb{P}^{3}$, which can be given by
$$
x^{2}y+xz^{2}+zt^{2}+tx^{2}=0\subset\mathbb{P}^{3}\cong\mathrm{Proj}\Big(\mathbb{C}[x,y,z,t]\Big),
$$

\item $\mathrm{lct}(X,\mathrm{Aut}(X))=4$ if $X$ is the Fermat cubic surface in $\mathbb{P}^{3}$,%

\item $\mathrm{lct}(X,\mathrm{Aut}(X))=2$ if $X$ is the~blow up of $\mathbb{P}^{2}$ at four general points.%
\end{itemize}
\end{example}

There exists a~complete solution to Problem~\ref{problem:II} for
$\mathbb{P}^{2}$ (cf. Theorem~\ref{theorem:smooth-quadric}).

\begin{example}[{\cite{MarPr99}, \cite{ChSh09a}}]
\label{example:complete-answer} Suppose that
$X\cong\mathbb{P}^{2}$. Then the following are equivalent:
\begin{itemize}
\item the inequality $\mathrm{lct}(X,G)>1$ holds, %
\item the inequality $\mathrm{lct}(X,G)\geqslant 4/3$ holds,%
\item there are no $G$-invariant curves in $|L|$, $|2L|$, $|3L|$, where $L$ is a line on $\mathbb{P}^{2}$,%
\item the subgroup $G$ is conjugate to one of the following subgroups:%
\begin{itemize}
\item the subgroup isomorphic to $\mathbb{PSL}(2,\mathbb{F}_{7})$
that leaves invariant th~quartic curve
$$
x^3y+y^3z+z^3x=0\subset\mathbb{P}^2\cong\mathrm{Proj}\Big(\mathbb{C}[x,y,z]\Big),%
$$

\item the subgroup isomorphic to $\mathbb{A}_{6}$ that leaves
invariant the sextic curve
$$
10x^{3}y^{3}+9zx^{5}+9zy^{5}+27z^{6}=45x^{2}y^{2}z^{2}+135xyz^{4}\subset\mathbb{P}^{2}\cong\mathrm{Proj}\Big(\mathbb{C}[x,y,z]\Big),
$$
\item the~Hessian subgroup of order $648$ (see \cite{YauYu93}),
\item an~index $3$ subgroup of the~Hessian subgroup.%
\end{itemize}
\end{itemize}
\end{example}

In this paper, we prove the following result, which solves
Problem~\ref{problem:I}.

\begin{theorem}
\label{theorem:main} Suppose that $\mathrm{dim}(X)=2$. Then
the~following~are~equivalent:
\begin{itemize}
\item there exists a~finite subgroup $G\subset\mathrm{Aut}(X)$ such that $\mathrm{lct}(X,G)>1$,%
\item one of the~following cases hold:
\begin{itemize}
\item either $X\cong\mathbb{P}^{2}$ or $X\cong\mathbb{P}^{1}\times \mathbb{P}^{1}$,%
\item or $\mathrm{Aut}(X)$ is finite and $X$ is one of the~following surfaces:%
\begin{itemize}
\item a~sextic surface in $\mathbb{P}(1,1,2,3)$ such that $\mathrm{Aut}(X)$ is not Abelian%
\item a~quartic surface in $\mathbb{P}(1,1,1,3)$ such that
$$
\mathrm{Aut}\big(X\big)\in\Big\{\mathbb{S}_4\times\mathbb{Z}_2, \big(\mathbb{Z}_4^2\rtimes\mathbb{S}_3\big)\times\mathbb{Z}_2, \mathbb{PSL}\big(2,\mathbb{F}_7\big)\times\mathbb{Z}_2\Big\},%
$$%
\item either the Clebsch cubic surface or the Fermat cubic surface in $\mathbb{P}^3$,%
\item an~intersection of two quadrics in $\mathbb{P}^4$ such that $\mathrm{Aut}(X)\in\{\mathbb{Z}_2^4 \rtimes\mathbb{S}_3,\mathbb{Z}_2^4\rtimes\mathbb{D}_{5}\}$,%
\item the~blow of $\mathbb{P}^{2}$ at four general points.
\end{itemize}
\end{itemize}
\end{itemize}
\end{theorem}

\begin{proof}
This follows from Examples~\ref{example:partial-answer} and
\ref{example:complete-answer},
Corollaries~\ref{corollary:binary-groups},
\ref{corollary:dP1-main}, \ref{corollary:dP2-main},
\ref{corollary:dP3-K}, \ref{corollary:dP4-2K} and
\ref{corollary:big-degree}.
\end{proof}

\begin{corollary}
\label{corollary:main-I} If $\mathrm{dim}(X)=2$ and
$\mathrm{Aut}(X)$ is finite, then the~following~are~equivalent:
\begin{itemize}
\item the inequality $\mathrm{lct}(X,\mathrm{Aut}(X))>1$ holds,%
\item the linear system $|-K_{X}|$ contains no $\mathrm{Aut}(X)$-invariant curves.%
\end{itemize}
\end{corollary}

The proof of Theorem~\ref{theorem:main} is based on auxiliary
results (see Theorems~\ref{theorem:dP1-lct-lct1},
\ref{theorem:dP1-lct-lct2}, \ref{theorem:dP2-lct-lct1},
\ref{theorem:dP2-lct-lct2}~and~\ref{theorem:dP3-lct-lct1}) that
can be used to explicitly compute the number $\mathrm{lct}(X,G)$
in many cases.

\begin{example}
\label{example:dP1-S4} Let $X$ be a sextic hypersurface in
$\mathbb{P}(1,1,2,3)$ that is given by
$$
t^2=z^3+xy\big(x^4-y^4\big)\subset\mathbb{P}\big(1,1,2,3\big)\cong\mathrm{Proj}\Big(\mathbb{C}[x,y,z,t]\Big),%
$$
where $\mathrm{wt}(x)=\mathrm{wt}(y)=1$, $\mathrm{wt}(z)=2$,
$\mathrm{wt}(t)=3$. Then
$\mathrm{Aut}(X)\cong\mathbb{Z}_3\times\mathbb{Z}_{2\bullet}\mathbb{S}_4$,
which implies that
$$
\mathrm{lct}\Big(X,\mathrm{Aut}\big(X\big)\Big)=\mathrm{lct}_{2}\Big(X,\mathrm{Aut}\big(X\big)\Big)=\frac{5}{3}
$$
by Theorems~\ref{theorem:main} and \ref{theorem:dP1-lct-lct2},
since there is a~$\mathrm{Aut}(X)$-invariant cuspidal curve in
$|-2K_{X}|$.
\end{example}

We decided not to solve Problem~\ref{problem:II} in this paper as
the~required amount of computations~is too big (a priori this can
be done using Theorem~\ref{theorem:main} and
Theorem~\ref{theorem:smooth-quadric}).

\begin{example}
\label{example:dP5} Suppose that $X$ is the~blow of
$\mathbb{P}^{2}$ at four general points. Then
$\mathrm{Aut}(X)\cong\mathbb{S}_{5}$~and
$$
\mathrm{lct}\big(X,G\big)>1\iff \mathrm{lct}\big(X,G\big)=2\iff |G|\in\big\{60,120\big\},%
$$
since it easily follows from Example~\ref{example:GAFA},
Corollary~\ref{corollary:weakly-exceptional},
\cite[Lemma~5.7]{Ch07b} and \cite[Lemma~5.8]{Ch07b} that
$$
\mathrm{lct}\big(X,G\big)=\left\{%
\aligned
&2\ \text{if}\ G\cong\mathbb{S}_5,\\
&2\ \text{if}\ G\cong\mathbb{A}_5,\\
&1\ \text{if}\ G\cong\mathbb{Z}_5\rtimes \mathbb{Z}_4,\\
&4/5\ \text{if}\ G\cong\mathbb{D}_{5},\\
&4/5\ \text{if}\ G\cong\mathbb{Z}_5,\\
&1/2\ \text{if $G$ is a trivial group}.\\
\endaligned\right.%
$$
\end{example}

Note that the number $\mathrm{lct}(X,G)$ plays an important role
in Birational Geometry (see \cite{ChSh08c},~\cite{Ch07b}), but we
decided not to discuss birational applications of
Theorem~\ref{theorem:main} in this paper.

\section{Preliminaries}
\label{section:preliminaries}

Let $X$ be a smooth surface, and let $D$ be an effective
$\mathbb{Q}$-divisor on $X$.~Put
$$
D=\sum_{i=1}^{r}a_{i}D_{i},
$$
where $D_{i}$ is an irreducible curve, and $a_{i}\in\mathbb{Q}$
such that $a_{i}\geqslant 0$. Suppose that $B_{i}\neq B_{j}$ for
$i\neq j$.

Let $\pi\colon\bar{X}\to X$ be a birational morphism such that
$\bar{X}$ is smooth as well. Put
$\bar{D}=\sum_{i=1}^{r}a_{i}\bar{D}_{i}$, where $\bar{D}_{i}$ is a
proper transform of the~curve $D_{i}$ on the~surface $\bar{X}$.
Then
$$
K_{\bar{X}}+\bar{D}\sim_{\mathbb{Q}}\pi^{*}\Big(K_{X}+D\Big)+\sum_{i=1}^{n}c_{i}E_{i},
$$
where $c_{i}\in\mathbb{Q}$ and $E_{i}$ is a $\pi$-exceptional
curve. Suppose that
$\sum_{i=1}^{r}\bar{D}_{i}+\sum_{i=1}^{n}E_{i}$
is~a~s.n.c.~divisor.

\begin{definition}
\label{definition:log canonical-singularities} The log pair $(X,
D)$ is KLT (respectively, log canonical) if
\begin{itemize}
\item the~inequality $a_{i}<1$ holds (respectively, the~inequality $a_{i}\leqslant 1$ holds),%
\item the~inequality $c_{j}>-1$ holds (respectively, the~inequality $c_{j}\geqslant -1$ holds),%
\end{itemize}
for every $i\in\{1,\ldots,r\}$ and $j\in\{1,\ldots,n\}$.
\end{definition}

We say that $(X, D)$ is strictly log canonical if $(X,D)$ is log
canonical and not KLT.

\begin{remark}
\label{remark:log-pull-back} The log pair $(X,D)$ is KLT $\iff$
the~log pair $(\bar{X},\bar{D}-\sum_{i=1}^{n}c_{i}E_{i})$ is KLT.
\end{remark}

Note that Definition~\ref{definition:log canonical-singularities}
has local nature and it does not depend on the~choice of $\pi$.

\begin{remark}
\label{remark:convexity}  Let $\hat{D}$ be an~effective
$\mathbb{Q}$-divisor on the~surface $X$ such that $(X,\hat{D})$ is KLT and
$$
\hat{D}=\sum_{i=1}^{r}\hat{a}_{i}D_{i}\sim_{\mathbb{Q}} D,%
$$
where $\hat{a}_{i}$ is a non-negative rational number. Suppose
that $(X,D)$ is not KLT. Put
$$
\alpha=\mathrm{min}\Bigg\{\frac{a_{i}}{\hat{a}_{i}}\ \Big\vert\ \hat{a}_{i}\ne 0\Bigg\},%
$$
where $\alpha$ is well defined and $\alpha<1$, since $(X,D)$ is
not KLT. Put
$$
D^{\prime}=\sum_{i=1}^{r}\frac{a_{i}-\alpha\hat{a}_{i}}{1-\alpha}D_{i}\sim_{\mathbb{Q}} \hat{D}\sim_{\mathbb{Q}} D,%
$$
and choose $k\in\{1,\ldots,r\}$ such that
$\alpha=a_{k}/\hat{a}_{k}$. Then
$D_{k}\not\subset\mathrm{Supp}(D^{\prime})$ and $(X,D^{\prime})$
is not KLT.
\end{remark}

Let $P$ be a point of the~surface $X$. Recall that $X$ is smooth
by assumption. Then
$$
\mathrm{mult}_{P}\big(D\big)\geqslant 2\Longrightarrow P\in\mathrm{LCS}\big(X,D\big)\Longrightarrow\mathrm{mult}_{P}\big(D\big)\geqslant 1.%
$$

\begin{example}
\label{example:three-conics} If $r=4$, $a_{1}=1/2$,
$a_{2}=a_{3}=a_{4}=2/5$ and
$$
3\geqslant \mathrm{mult}_{P}\Big(D_{2}\cdot D_{1}\Big)\geqslant 2=\mathrm{mult}_{P}\Big(D_{3}\cdot D_{1}\Big)\geqslant\mathrm{mult}_{P}\Big(D_{4}\cdot D_{1}\Big)=1,%
$$
then the~log pair $(X,D)$ is log canonical at the~point $P\in X$.
\end{example}

The set of non-KLT points of the~log pair $(X,D)$ is denoted by
$\mathrm{LCS}(X,D)$. Put
$$
\mathcal{I}\big(X, D\big)=\pi_{*}\Bigg(\sum_{i=1}^{n}\lceil c_{i}\rceil E_{i}-\sum_{i=1}^{r}\lfloor a_{i}\rfloor D_{i}\Bigg),%
$$
and let $\mathcal{L}(X, D)$ be a~subscheme that corresponds to
the~ideal sheaf $\mathcal{I}(X, D)$. Then
$$
\mathrm{LCS}\big(X, D\big)=\mathrm{Supp}\Big(\mathcal{L}\big(X, D\big)\Big).%
$$

\begin{theorem}[{\cite[Theorem~9.4.8]{La04}}]
\label{theorem:Shokurov-vanishing} Let $H$ be a~nef and big
$\mathbb{Q}$-divisor on $X$~such~that
$$
K_{X}+D+H\equiv L
$$
for some Cartier divisor $L$ on the~surface $X$. Then
$H^{1}(\mathcal{I}(X, D)\otimes\mathcal{O}_{X}(L))=0$.
\end{theorem}

Let $\eta\colon X\to Z$ be a~surjective morphism with connected
fibers.

\begin{theorem}[{\cite[Theorem~7.4]{Ko97}}]
\label{theorem:connectedness} Let $F$ be a~fiber of the~morphism
$\eta$. Then the~locus
$$
\mathrm{LCS}\Big(X, D\Big)\cap F
$$
is connected if $-(K_{X}+D)$ is $\eta$-nef and $\eta$-big.
\end{theorem}

\begin{corollary}
\label{corollary:connectedness} If $-(K_{X}+D)$ is ample, then
$\mathrm{LCS}(X,D)$ is connected.
\end{corollary}

Recall that $\mathcal{I}(X, D)$ is known as the~multiplier ideal
sheaf (see \cite[Section~9.2]{La04}).

\begin{lemma}[{\cite[Theorem~7.5]{Ko97}}]
\label{lemma:adjunction} Suppose that the~log pair $(X,D)$ is KLT
in a~punctured neighborhood of the~point $P$, but the~log pair
$(X,D)$ is not KLT at the~point $P$. Then
$$
\Bigg(\sum_{i=2}^{r}a_{i}D_{i}\Bigg)\cdot D_{1}>1
$$
in the~case when $P\in D_{1}\setminus\mathrm{Sing}(D_{1})$.
\end{lemma}

Recall that it follows from Definition~\ref{definition:log
canonical-singularities} that if the log pair $(X,D)$ is KLT in
a~punctured neighborhood of the~point $P\in X$, then $a_{i}<1$ for
every $i\in\{1,\ldots,r\}$.

\begin{theorem}[{\cite[Theorem~1.28]{ChKo10}}]
\label{theorem:Cheltsov-Kosta} In the assumptions and notation of
Lemma~\ref{lemma:adjunction}, suppose~that
$$
P\in\Big(D_{1}\setminus\mathrm{Sing}\big(D_{1}\big)\Big)\bigcap\Big(D_{2}\setminus\mathrm{Sing}\big(D_{2}\big)\Big)
$$
and the curve $D_{1}$ intersects the curve $D_{2}$ transversally
at the point $P\in X$. Then
$$
\Bigg(\sum_{i=3}^{r}a_{i}D_{i}\Bigg)\cdot D_{1}\geqslant M+Aa_{1}-a_{2}\ \text{or}\ \Bigg(\sum_{i=3}^{r}a_{i}D_{i}\Bigg)\cdot D_{2}\geqslant N+Ba_{2}-a_{1}%
$$
for some non-negative rational numbers $A,B,M,N,\alpha,\beta$ that
satisfy the following conditions:
\begin{itemize}
\item $\alpha a_{1}+\beta a_{2}\leqslant 1$ and $A(B-1)\geqslant 1\geqslant\mathrm{max}(M,N)$,%
\item $\alpha(A+M-1)\geqslant A^{2}(B+N-1)\beta$ and  $\alpha(1-M)+A\beta\geqslant A$,%
\item either $2M+AN\leqslant 2$ or
$\alpha(B+1-MB-N)+\beta(A+1-AN-M)\geqslant AB-1$.%
\end{itemize}
\end{theorem}

\begin{corollary}
\label{corollary:Cheltsov-Kosta} In the assumptions and notation
of Theorem~\ref{theorem:Cheltsov-Kosta}, if $6a_{1}+a_{2}<4$, then
$$
\Bigg(\sum_{i=3}^{r}a_{i}D_{i}\Bigg)\cdot D_{1}>2a_{1}-a_{2}\ \text{or}\ \Bigg(\sum_{i=3}^{r}a_{i}D_{i}\Bigg)\cdot D_{2}> 1+\frac{3}{2}a_{2}-a_{1}.%
$$
\end{corollary}

Let $\sigma\colon\tilde{X}\to X$ be a blow up of the~point $P$,
and let $F$ be the~$\sigma$-exceptional curve. Then
$$
K_{\tilde{X}}+\tilde{D}\sim_{\mathbb{Q}}\sigma^{*}\big(K_{X}+D\big)+\Big(1-\mathrm{mult}_{P}\big(D\big)\Big)F
$$
where $\tilde{D}$ is the~proper transform of the~divisor $D$ on
the~surface $\tilde{X}$.

\begin{remark}
\label{remark:blow-up-inequality}  Suppose that
$\mathrm{mult}_{P}(D)<2$, the~log pair  $(X,D)$ is KLT in
a~punctured neighborhood of the~point~$P$, and $(X,D)$ is not KLT
at the~point $P$. Then there is a~point $Q\in F$~such~that
$$
\mathrm{LCS}\Big(\tilde{X},\ \tilde{D}+\big(\mathrm{mult}_{P}\big(D\big)-1\big)F\Big)\cap F=Q%
$$
by Theorem~\ref{theorem:connectedness}, which implies that
$\mathrm{mult}_{Q}(\tilde{D})+\mathrm{mult}_{P}(D)\geqslant 2$.
\end{remark}

Suppose that $X$ is a~smooth del Pezzo surface and
$D\sim_{\mathbb{Q}} -\lambda K_{X}$ for some
$\lambda\in\mathbb{Q}$.

\begin{lemma}
\label{lemma:Nadel-IILC} Suppose that $\mathrm{LCS}(X,D)$ is a
non-empty finite set. Then
$$
\big|\mathrm{LCS}(X,D)\big|\leqslant h^0\Big(X,\mathcal{O}_X\big(-\lceil\lambda-1\rceil K_X\big)\Big)%
$$
and for every point $P\in \mathrm{LCS}(X,D)$ there exists a~curve
$C\in |-\lceil\lambda-1\rceil K_X|$ such that
$$
\mathrm{LCS}\big(X,D\big)\setminus P\subset\mathrm{Supp}\big(C\big)\not\ni P.%
$$
\end{lemma}

\begin{proof}
The required assertions follow from
Theorem~\ref{theorem:Shokurov-vanishing}.
\end{proof}

Let $G$ be a finite subgroup in $\mathrm{Aut}(X)$ such that the
following two conditions are satisfied:
\begin{itemize}
\item a $G$-invariant subgroup of the~group $\mathrm{Pic}(X)$ is generated by $-K_{X}$,%
\item the divisor $D$ is $G$-invariant.%
\end{itemize}

\begin{remark}
\label{remark:abelian-groups} If $G$ is Abelian, then
$\mathrm{lct}(X,G)\leqslant 1$.
\end{remark}

Let $\xi$ be the~smallest integer such that $|-\xi K_X|$ contains
a $G$-invariant curve.

\begin{lemma}
\label{lemma:LCS-zero-dimensional} If $\xi>\lambda$, then
$\mathrm{LCS}(X,D)$ is zero-dimensional.
\end{lemma}

\begin{proof}
Suppose that $\mathrm{LCS}(X,D)$ is not zero-dimensional. Then
$$
D=\gamma B+D^{\prime},
$$
where $B$ is a~$G$-invariant effective Weil divisor on $X$,
$\gamma$ is a rational number such that $\gamma\geqslant 1$ and
$D^{\prime}$ is a~$G$-invariant effective $\mathbb{Q}$-divisor $D$
on the~surface $X$. We have that
$$
B\sim -nK_{X}
$$
for some positive integer $n$ such that $n\geqslant\xi$. Thus, we
see that
$$
\lambda\Big(-K_{X}\Big)^{2} =-K_{X}\cdot D=\gamma \Big(-K_{X}\cdot B\Big)+\Big(-K_{X}\cdot D^{\prime}\Big)\geqslant\gamma \Big(-K_{X}\cdot B\Big)=n\gamma\Big(-K_{X}\Big)^{2}\geqslant \xi\Big(-K_{X}\Big)^{2},%
$$
which implies that $\xi\leqslant \lambda$.
\end{proof}

\begin{corollary}
\label{corollary:GAFA} Let $k$ be the~length of the~smallest
$G$-orbit in $X$. Then $\mathrm{lct}(X,G)=\xi$ if
$$
h^0\Big(X,\mathcal{O}_X\Big(\big(1-\xi\big)K_{X}\Big)\Big)<k.
$$
\end{corollary}

\begin{corollary}
\label{corollary:weakly-exceptional} If $X$ does not contain
$G$-fixed points, then $\mathrm{lct}(X,G)\geqslant 1$.
\end{corollary}

Most of results described in this section are valid in more
general settings (see \cite{Ko97}).

\section{Double quadric cone}
\label{sec:degree-one}

Let $X$ be a~smooth sextic surface in $\mathbb{P}(1,1,2,3)$. Then
$X$ can be given by an~equation
$$
t^2=z^3+zf_4\big(x,y\big)+f_6\big(x,y\big)\subset\mathbb{P}\big(1,1,2,3\big)\cong\mathrm{Proj}\Big(\mathbb{C}[x,y,z,t]\Big),%
$$
where $\mathrm{wt}(x)=\mathrm{wt}(y)=1$, $\mathrm{wt}(z)=2$,
$\mathrm{wt}(t)=3$, and $f_i(x,y)$ is a form of degree $i$.

\begin{remark}
\label{remark:dP1-smoothness} It follows from the smoothness of
the surface $X$ that
\begin{itemize}
\item a~common root of the~forms $f_4(x,y)$ and $f_6(x,y)$ is not a~multiple root of the~form~$f_6(x,y)$,%
\item the~form $f_6(x,y)$ is not a zero form.%
\end{itemize}
\end{remark}

Let $\tau$ be the~involution in $\mathrm{Aut}(X)$ such that
$\tau([x:y:z:t])=[x:y:z:-t]$.

\begin{lemma}[{\cite[Lemma~6.18]{DoIs06}}]
\label{lemma:dP1-Bertini} A $\tau$-invariant subgroup in
$\mathrm{Pic}(X)$ is generated by $-K_{X}$.
\end{lemma}

Let $G$ be a subgroup in $\mathrm{Aut}(X)$ such that $\tau\in G$.
Recall that $\mathrm{Aut}(X)$ is finite.

\begin{lemma}
\label{lemma:dP1-2K} There exists a~$G$-invariant curve in
$|-2K_X|$.
\end{lemma}

\begin{proof}
Let $C$ be the~curve on $X$ that is cut out by $z=0$. Then $C$ is
$G$-invariant.
\end{proof}

\begin{corollary}
\label{corollary:dP1-2K} The inequality
$\mathrm{lct}(X,G)\leqslant 2$ holds.
\end{corollary}

The main purpose of this section is to prove the~following two
results.

\begin{theorem}
\label{theorem:dP1-lct-lct1} Suppose that there exists
a~$G$-invariant curve in $|-K_X|$. Then
$$
\mathrm{lct}\big(X,G\big)=\mathrm{lct}_1\big(X,G\big)\in\big\{5/6,1\big\}.
$$
\end{theorem}

\begin{proof}
If $\mathrm{lct}_1(X,G)=5/6$, then $\mathrm{lct}(X,G)=5/6$ by
Example~\ref{example:GAFA}, since
$\mathrm{lct}_1(X,G)\in\{5/6,1\}$.

Suppose that  $\mathrm{lct}(X,G)<\mathrm{lct}_1(X,G)=1$. Let us
derive a contradiction.

There exists a~$G$-invariant effective $\mathbb{Q}$-divisor $D$ on
the~surface $X$ such that
$$
D\sim_{\mathbb{Q}} -K_{X}
$$
and the~log pair $(X,\lambda D)$ is strictly log canonical for
some rational number $\lambda<\mathrm{lct}_1(X,G)$.

By Theorem~\ref{theorem:connectedness} and
Lemma~\ref{lemma:Nadel-IILC}, the~locus $\mathrm{LCS}(X,\lambda
D)$ consists of a single point $P\in X$ such that $P$ is not
the~base point of the~pencil $|-K_{X}|$. Then $P$ is
$G$-invariant.

Let $C$ be the~unique curve in the~pencil $|-K_{X}|$ that passes
through $P$. Then $C$ is $G$-invariant, and we may assume that
$C\not\subseteq\mathrm{Supp}(D)$ (see
Remark~\ref{remark:convexity}). Then
$$
1>\lambda= \lambda D \cdot C\geqslant \lambda\mathrm{mult}_P\big(D\big)>1,%
$$
which is a contradiction.
\end{proof}

\begin{theorem}
\label{theorem:dP1-lct-lct2} Suppose that there are no
$G$-invariant curves in $|-K_X|$. Then
$$
1\leqslant\mathrm{lct}\big(X,G\big)=\mathrm{lct}_2\big(X,G\big)\leqslant 2.%
$$
\end{theorem}

\begin{proof}
Arguing as in the~proof of Theorem~\ref{theorem:dP1-lct-lct1}, we
see that $\mathrm{lct}(X,G)\geqslant 1$. Then
$$
1\leqslant\mathrm{lct}\big(X,G\big)\leqslant\mathrm{lct}_2\big(X,G\big)\leqslant 2%
$$
by Corollary~\ref{corollary:dP1-2K}. Suppose that
$\mathrm{lct}(X,G)<\mathrm{lct}_2(X,G)$. Let us derive a
contradiction.

There exists a~$G$-invariant effective $\mathbb{Q}$-divisor $D$ on
the~surface $X$ such that
$$
D\sim_{\mathbb{Q}} -K_{X}
$$
and $(X,\lambda D)$ is strictly log canonical for some rational
number $\lambda<\mathrm{lct}_2(X,G)$.

By Lemmata~\ref{lemma:LCS-zero-dimensional},
\ref{lemma:Nadel-IILC} and \ref{lemma:dP1-Bertini}, the~locus
$\mathrm{LCS}(X,\lambda D)\ne\varnothing$ consists of exactly two
points, which are different from the~base point of the~pencil
$|-K_X|$.

Let $P_{1}$ and $P_{2}$ be two points in $\mathrm{LCS}(X, \lambda
D)$. Then
$$
\mathrm{mult}_{P_{1}}\big(D\big)=\mathrm{mult}_{P_{2}}\big(D\big)\geqslant\frac{1}{\lambda}>\frac{1}{2}.
$$

Let $C_{1}$ and $C_{2}$ be the~curves in $|-K_{X}|$ such that
$P_{1}\in C_{1}$ and $P_{2}\in C_{2}$. Then
$$
C_{1}\ne C_{2}
$$
by Lemma~\ref{lemma:LCS-zero-dimensional}. Note that $C_{1}+C_{2}$
is $G$-invariant and $C_{1}+C_{2}\sim -2K_{X}$.

By Remark~\ref{remark:convexity}, we may assume that $C_1$ and
$C_2$ are not contained in $\mathrm{Supp}(D)$. Then
$$
2=D\cdot\Big(C_1+C_{2}\Big)\geqslant\sum_{i=1}^{2}\mathrm{mult}_{P_i}\big(D\big)\mathrm{mult}_{P_i}\big(C_i\big)\geqslant 2\mathrm{mult}_{P_1}\big(D\big)=2\mathrm{mult}_{P_2}\big(D\big)>1,%
$$
which implies that
$\mathrm{mult}_{P_1}(D)=\mathrm{mult}_{P_2}(D)\leqslant 1$ and
$\mathrm{mult}_{P_1}(C_{1})=\mathrm{mult}_{P_2}(C_{2})=1$.

Let $\sigma\colon\bar{X}\to X$ be the~blow-up of the~surface $X$
at the~points $P_1$ and $P_2$, let $E_{1}$ and $E_{2}$ be
the~exceptional curves of the~morphism $\sigma$ such that
$\sigma(E_{1})=P_{1}$ and $\sigma(E_{2})=P_{2}$. Then
$$
K_{\bar{X}}+\lambda\bar{D}+\Big(\lambda\mathrm{mult}_{P_1}\big(D\big)-1\Big)E_1+\Big(\lambda\mathrm{mult}_{P_2}\big(D\big)-1\Big)E_2\sim_{\mathbb{Q}}\sigma^{*}\Big(K_X + \lambda D\Big),%
$$
where $\bar{D}$ is the~proper transform of the~divisor $D$ on
the~surface $\bar{X}$.

It follows from Remark~\ref{remark:blow-up-inequality} that there
are points $Q_{1}\in E_{1}$ and $Q_{2}\in E_{2}$ such that
$$
\mathrm{LCS}\Bigg(\bar{X},\lambda\bar{D}+\Big(\lambda\mathrm{mult}_{P_1}\big(D\big)-1\Big)E_1+\Big(\lambda\mathrm{mult}_{P_2}\big(D\big)-1\Big)E_2\Bigg)=\Big\{Q_{1},Q_{2}\Big\},
$$
as
$\lambda\mathrm{mult}_{P_1}(D)-1=\lambda\mathrm{mult}_{P_2}(D)-1<1$.
By Remark~\ref{remark:blow-up-inequality}, we have
\begin{equation}
\label{equation:dP1-inequality-I}
\mathrm{mult}_{P_{1}}\big(D\big)+\mathrm{mult}_{Q_{1}}\big(\bar{D}\big)=\mathrm{mult}_{P_{2}}\big(D\big)+\mathrm{mult}_{Q_{2}}\big(\bar{D}\big)\geqslant\frac{2}{\lambda}>1.
\end{equation}

Note that the~action of the~group $G$ on the~surface $X$ naturally
lifts to an action on $\bar{X}$.

Let $\bar{C}_{1}$ and $\bar{C}_{2}$ be the~proper transforms of
the~curves $C_{1}$ and $C_{2}$ on the~surface $\bar{X}$,
respectively.~Then
$$
1-\mathrm{mult}_{P_{1}}\big(D\big)=\bar{C}_{1}\cdot\bar{D}\geqslant\mathrm{mult}_{Q_{1}}\big(\bar{C}_{1}\big)\mathrm{mult}_{Q_{1}}\big(\bar{D}\big),
$$
which implies that $Q_{1}\not\in\bar{C}_{1}$ by
$(\ref{equation:dP1-inequality-I})$. Similarly, we see that
$Q_{2}\not\in\bar{C}_{2}$.

Let $R$ be a curve that is cut out on $X$ by $t=0$. Then $P_{1}\in
R\ni P_{2}$, since $\tau\in G$.

Let $\bar{R}$ be the~proper transform of the~curve $R$ on
the~surface $\bar{X}$.~Then
$$
Q_{1}=\bar{R}\cap E_{1},
$$
since $\bar{R}\cap E_{1}$ and $\bar{C}_{1}\cap E_{1}$ are the only
$\tau$-fixed points in $E_{1}$. Similarly, we see that
$Q_{2}=\bar{R}\cap E_{2}$.

By Remark~\ref{remark:convexity}, we may assume that
$\bar{R}\not\subseteq\mathrm{Supp}(\bar{D})$, since $R$ is smooth.
Then
$$
\mathrm{mult}_{Q_1}\big(\bar{D}\big)+\mathrm{mult}_{Q_2}\big(\bar{D}\big)\leqslant\bar{D}\cdot\bar{R}=3-\mathrm{mult}_{P_1}\big(D\big)-\mathrm{mult}_{P_2}\big(D\big),%
$$
which implies that
$\mathrm{mult}_{Q_1}(\bar{D})+\mathrm{mult}_{P_1}(D)=\mathrm{mult}_{Q_2}(\bar{D})+\mathrm{mult}_{P_2}(D)\leqslant
3/2$. Thus, we have
\begin{equation}
\label{equation:dP1-inequality-II}
\frac{3}{2}\geqslant\mathrm{mult}_{Q_1}\big(\bar{D}\big)+\mathrm{mult}_{P_1}\big(D\big)=\mathrm{mult}_{Q_2}\big(\bar{D}\big)+\mathrm{mult}_{P_2}\big(D\big)\geqslant\frac{2}{\lambda}>1.%
\end{equation}

The linear system $|-2K_{X}|$ induces a double cover $\pi\colon
X\to Q$ that is branched over $\pi(R)$, where $Q$ is an
irreducible quadric cone in $\mathbb{P}^{3}$. Let $\Pi_{1}$ and
$\Pi_{2}$ be the~planes in $\mathbb{P}^{3}$ such that
$$
\pi\big(P_{1}\big)\in\Pi_{1}\cap\Pi_{2}\ni \pi\big(P_{2}\big),
$$
the plane $\Pi_{1}$ is tangent to $\pi(R)$ at $\pi(P_1)$ and
$\Pi_{2}$ is tangent to $\pi(R)$ at   $\pi(P_2)$. Then
$$
\Pi_{1}\not\ni\mathrm{Sing}\big(Q\big)\not\in\Pi_{2},
$$
since $C_{1}$ and $C_{2}$ are smooth at $P_{1}$ and $P_{2}$
respectively. Then $\Pi_{1}\cap Q$ and $\Pi_{2}\cap Q$ are smooth.

Let $Z_{1}$ and $Z_{2}$ be curves in $|-2K_{X}|$ such that
$\pi(Z_{1})=\Pi_{1}\cap Q$ and  $\pi(Z_{2})=\Pi_{2}\cap Q$. Then
$$
Z_{1}+Z_{2}\in \big|-4K_{X}\big|
$$
and the~curve $Z_{1}+Z_{2}$ is $G$-invariant. Note that the~case
$Z_{1}=Z_{2}$ is also possible.

Suppose that $Z_{1}=Z_{2}$. It follows from
Remark~\ref{remark:convexity} that we may assume that
$Z_{1}\not\subset\mathrm{Supp}(D)$, as we have $Z_{1}\in
|-2K_{X}|$. It should be mentioned (we need this for
Corollary~\ref{corollary:dP1-lct-lct2-5-3}) that either
$$
\Bigg(X, \frac{5}{6} Z_{1}\Bigg)
$$
is strictly log canonical or the~log pair $(X,Z_{1})$ is strictly
log canonical. Then
$$
2=Z_1\cdot D\geqslant\mathrm{mult}_{P_1}\big(Z_{1}\big)\mathrm{mult}_{P_1}\big(D\big)+\mathrm{mult}_{P_2}\big(Z_{1}\big)\mathrm{mult}_{P_2}\big(D\big)\geqslant 2\mathrm{mult}_{P_1}\big(D\big)+2\mathrm{mult}_{P_2}\big(D\big)\geqslant\frac{4}{\lambda}>2,%
$$
by $(\ref{equation:dP1-inequality-I})$. The obtained contradiction
implies that $Z_{1}\ne Z_{2}$.

Note that
$\mathrm{mult}_{P_{1}}(Z_{1}+Z_{2})=\mathrm{mult}_{P_{1}}(Z_{1}+Z_{2})=3$
by construction. Suppose that
\begin{equation}
\label{equation:dP1-special-log-pair}
\Bigg(X,\frac{\lambda}{4}\Big(Z_1 + Z_2\Big)\Bigg)
\end{equation}
is KLT. By Remark~\ref{remark:convexity}, we may assume that
$\mathrm{Supp}(D)\cap Z_{1}$ and $\mathrm{Supp}(D)\cap Z_{2}$ are
finite subsets.

Let $\bar{Z}_{1}$ and $\bar{Z}_{2}$ be the~proper transforms of
the~curves $Z_{1}$ and $Z_{2}$ on the~surface $\bar{X}$,
respectively.~Then
$$
0\leqslant \bar{D}\cdot\Big(\bar{Z_1}+\bar{Z_2}\Big)=4-3\Big(\mathrm{mult}_{P_1}\big(D\big)+\mathrm{mult}_{P_2}\big(D\big)\Big)=4-6\mathrm{mult}_{P_1}\big(D\big)=4-6\mathrm{mult}_{P_2}\big(D\big),%
$$
since
$\mathrm{mult}_{P_{1}}(Z_{1}+Z_{2})=\mathrm{mult}_{P_{2}}(Z_{1}+Z_{2})=3$.
Then
\begin{equation}
\label{equation:dP1-inequality-III}
\mathrm{mult}_{P_1}\big(D\big)=\mathrm{mult}_{P_2}\big(D\big)\leqslant\frac{2}{3}.
\end{equation}

Let  $\rho\colon\tilde{X} \to \bar{X}$ be a blow up of the~surface
$\bar{X}$ at the~points $Q_{1}$ and $Q_{2}$, let $F_{1}$ and
$F_{2}$ be the~exceptional curves of the~morphism $\rho$ such that
$\rho(F_{1})=Q_{1}$ and $\rho(F_{2})=Q_{2}$. Then
$$
K_{\tilde{X}}+\lambda\tilde{D}+\sum_{i=1}^2\Big(\lambda \mathrm{mult}_{P_i}\big(D\big)-1\Big)\tilde{E}_i+\sum_{i=1}^2\Big(\lambda \mathrm{mult}_{Q_i}\big(\bar{D}\big)+\lambda\mathrm{mult}_{P_i}\big(D\big)-2\Big)F_i \sim_{\mathbb{Q}}\big(\sigma\circ\rho\big)^{*}\Big(K_{X}+\lambda D\Big),%
$$
where $\tilde{D}$ and $\tilde{E}_{i}$ are proper transforms of
the~divisors $D$ and $E_{i}$ on the~surface $\tilde{X}$,
respectively.

It follows from Remark~\ref{remark:blow-up-inequality} that there
are points $O_{1}\in F_{1}$ and $O_{2}\in F_{2}$ such that
$$
\mathrm{LCS}\Bigg(\tilde{X}, \lambda\tilde{D}+\sum_{i=1}^2\Big(\lambda \mathrm{mult}_{P_i}\big(D\big)-1\Big)\tilde{E}_i+\sum_{i=1}^2\Big(\lambda \mathrm{mult}_{Q_i}\big(\bar{D}\big)+\lambda\mathrm{mult}_{P_i}\big(D\big)-2\Big)F_i\Bigg)=\Big\{O_{1},O_{2}\Big\},%
$$
as $\lambda
\mathrm{mult}_{Q_1}(\bar{D})+\lambda\mathrm{mult}_{P_1}(D)-2=\lambda
\mathrm{mult}_{Q_2}(\bar{D})+\lambda\mathrm{mult}_{P_2}(D)-2<1$ by
$(\ref{equation:dP1-inequality-II})$.

The~action of the~group $G$ on the~surface $\bar{X}$ naturally
lifts to an action on $\tilde{X}$ such that the~curves $F_{1}$
and $F_{2}$ contain exactly two points that are fixed by $\tau$,
respectively.

Let $\tilde{R}$ be the~proper transform of the~curve $R$ on
the~surface $\tilde{X}$. Then
\begin{itemize}
\item either $O_{1}=\tilde{E}_{1}\cap F_{1}$ and $O_{2}=\tilde{E}_{2}\cap F_{2}$,%
\item or $O_{1}=\tilde{R}\cap F_{1}$ and $O_{2}=\tilde{R}\cap F_{2}$.%
\end{itemize}

Suppose that $O_{1}=\tilde{E}_{1}\cap F_{1}$ and
$O_{2}=\tilde{E}_{2}\cap F_{2}$. It follows from
Lemma~\ref{lemma:adjunction} that
$$
2\lambda\mathrm{mult}_{P_1}\big(D\big)-2=\Bigg(\lambda\tilde{D}+\Big(\lambda\mathrm{mult}_{Q_1}\big(\bar{D}\big)+\lambda\mathrm{mult}_{P_1}\big(D\big)-2\Big)F_1\Bigg)\cdot\tilde{E}_1>1,%
$$
which implies that $\mathrm{mult}_{P_1}(D)>3/4$, which is
impossible by $(\ref{equation:dP1-inequality-III})$.

Thus, we see that $O_{1}=\tilde{R}\cap F_{1}$ and
$O_{2}=\tilde{R}\cap F_{2}$. Then
$$
\mathrm{LCS}\Bigg(\tilde{X}, \lambda\tilde{D}+\sum_{i=1}^2\Big(\lambda \mathrm{mult}_{Q_i}\big(\bar{D}\big)+\lambda\mathrm{mult}_{P_i}\big(D\big)-2\Big)F_i\Bigg)=\Big\{O_{1},O_{2}\Big\},%
$$
since $O_{1}\not\in\tilde{E}_{1}$ and
$O_{2}\not\in\tilde{E}_{2}$. Then it follows from
Remark~\ref{remark:blow-up-inequality} that
\begin{equation}
\label{equation:dP1-inequality-V}
\mathrm{mult}_{O_{1}}\big(\tilde{D}\big)+\mathrm{mult}_{Q_1}\big(\bar{D}\big)+\mathrm{mult}_{P_1}\big(D\big)=\mathrm{mult}_{O_{2}}\big(\tilde{D}\big)+\mathrm{mult}_{Q_2}\big(\bar{D}\big)+\mathrm{mult}_{P_2}\big(D\big)\geqslant\frac{3}{\lambda},%
\end{equation}
as $\lambda
\mathrm{mult}_{Q_i}(\bar{D})+\lambda\mathrm{mult}_{P_i}(D)-2\geqslant
0$ by $(\ref{equation:dP1-inequality-II})$. But
$$
3-\Big(\mathrm{mult}_{P_1}\big(D\big)+\mathrm{mult}_{P_2}\big(D\big)+\mathrm{mult}_{Q_1}\big(\bar{D}\big)+\mathrm{mult}_{Q_2}\big(\bar{D}\big)\Big)=\tilde{R}\cdot\tilde{D}\geqslant\mathrm{mult}_{O_1}\big(\tilde{D}\big)+\mathrm{mult}_{O_2}\big(\tilde{D}\big),%
$$
which contradicts $(\ref{equation:dP1-inequality-V})$, since
$\lambda<\mathrm{lct}_{2}(X,G)\leqslant 2$.

The obtained contradiction shows that
$(\ref{equation:dP1-special-log-pair})$ is not KLT.

It should be pointed out that we may apply all arguments we
already used for our~original log pair $(X,\lambda D)$ to the~log
pair $(\ref{equation:dP1-special-log-pair})$ with one exception:
we can not use $(\ref{equation:dP1-inequality-III})$.~Then
$$
\frac{3}{2}\geqslant\frac{\mathrm{mult}_{Q_1}\big(\bar{Z}_{1}+\bar{Z}_{2}\big)}{4}+\frac{\mathrm{mult}_{P_1}\big(Z_{1}+Z_{2}\big)}{4}=\frac{\mathrm{mult}_{Q_2}\big(\bar{Z}_{1}+\bar{Z}_{2}\big)}{4}+\frac{\mathrm{mult}_{P_2}\big(Z_{1}+Z_{2}\big)}{4}\geqslant\frac{2}{\lambda}>1%
$$
by $(\ref{equation:dP1-inequality-II})$. But
$\mathrm{mult}_{P_1}(Z_1+Z_2)=\mathrm{mult}_{P_2}(Z_1+Z_2)=3$.
Thus, we see that
$$
3\geqslant\mathrm{mult}_{Q_1}\big(\bar{Z}_{1}+\bar{Z}_{2}\big)=\mathrm{mult}_{Q_2}\big(\bar{Z}_{1}+\bar{Z}_{2}\big)\geqslant\frac{8}{\lambda}-3>1,%
$$
which implies that one of the~following two cases holds:
\begin{itemize}
\item either $\mathrm{mult}_{Q_1}(\bar{Z}_{1}+\bar{Z}_{2})=\mathrm{mult}_{Q_2}(\bar{Z}_{1}+\bar{Z}_{2})=2$,%
\item or $\mathrm{mult}_{Q_1}(\bar{Z}_{1}+\bar{Z}_{2})=\mathrm{mult}_{Q_2}(\bar{Z}_{1}+\bar{Z}_{2})=3$.%
\end{itemize}

It follows from the~construction of the~curves $Z_{1}$ and $Z_{2}$
that
$$
\bar{Z}_{2}\cap E_{1}=\bar{C}_{1}\cap E_{1}\ne Q_{1}\in\bar{R}\ni Q_{2}\ne \bar{C}_{2}\cap E_{2}=\bar{Z}_{1}\cap E_{2},%
$$
because $Z_{1}$ is smooth at the~point $P_{2}$ and $Z_{2}$ is
smooth at the~point $P_{1}$. Hence, we must have
$$
\mathrm{mult}_{Q_1}\big(\bar{Z}_{1}+\bar{Z}_{2}\big)=\mathrm{mult}_{Q_2}\big(\bar{Z}_{1}+\bar{Z}_{2}\big)=\mathrm{mult}_{Q_1}\big(\bar{Z}_{1}\big)=\mathrm{mult}_{Q_2}\big(\bar{Z}_{2}\big)=2,
$$
as $2=\mathrm{mult}_{P_1}(Z_{1})\geqslant
\mathrm{mult}_{Q_1}(\bar{Z}_{1})$ and
$2=\mathrm{mult}_{P_2}(Z_{2})\geqslant
\mathrm{mult}_{Q_2}(\bar{Z}_{2})$.

Let $\tilde{Z}_{i}$ be the~proper transforms of the~curve $Z_{i}$
on the~surface $\tilde{X}$. Then
$$
\varnothing\ne\mathrm{LCS}\Bigg(\tilde{X}, \frac{\lambda}{4}\Big(\tilde{Z}_{1}+\tilde{Z}_{2}\Big)+\frac{3\lambda-4}{4}\Big(\tilde{E}_{1}+\tilde{E}_{2}\Big)+\frac{5\lambda-8}{4}\Big(F_{1}+F_{2}\Big)\Bigg)\subsetneq F_{1}\cup F_{2},%
$$
since $3\lambda/4-1<1$ and $5\lambda/4-2<1$. On the~other hand, we
know that
$$
\tilde{Z}_{1}\cap\tilde{E}_{1}=\varnothing=\tilde{Z}_{2}\cap\tilde{E}_{2},$$
as we have
$\mathrm{mult}_{P_1}(Z_{1})=\mathrm{mult}_{Q_1}(\bar{Z}_{1})$ and
$\mathrm{mult}_{P_2}(Z_{2})=\mathrm{mult}_{Q_2}(\bar{Z}_{2})$.
Then
$$
\mathrm{LCS}\Bigg(\tilde{X}, \frac{\lambda}{4}\Big(\tilde{Z}_{1}+\tilde{Z}_{2}\Big)+\frac{5\lambda-8}{4}\Big(F_{1}+F_{2}\Big)\Bigg)=\Big\{\tilde{R}\cap F_{1},\tilde{R}\cap F_{2}\Big\}.%
$$

We can put $O_{1}=\tilde{R}\cap F_{1}$ and $O_{2}=\tilde{R}\cap
F_{i}$. Since $5\lambda/4-2\geqslant 0$, we must have
$$
\frac{\lambda}{4}\mathrm{mult}_{O_{1}}\big(\tilde{Z}_{1}\big)+\frac{5\lambda-8}{4}=\frac{\lambda}{4}\mathrm{mult}_{O_{2}}\big(\tilde{Z}_{2}\big)+\frac{5\lambda-8}{4}\geqslant 1,%
$$
which implies that $\mathrm{mult}_{O_{1}}(\tilde{Z}_{1})\geqslant
12/\lambda-5$ and $\mathrm{mult}_{O_{2}}(\tilde{Z}_{2})\geqslant
12/\lambda-5$. Whence
$$
2=\tilde{R}\cdot\Big(\tilde{Z}_{1}+\tilde{Z}_{2}\Big)\geqslant\mathrm{mult}_{O_{1}}\big(\tilde{Z}_{1}\big)+\mathrm{mult}_{O_{2}}\big(\tilde{Z}_{2}\big)\geqslant\frac{24}{\lambda}-10>2,%
$$
as $\lambda<2$. The obtained contradiction implies that
$(\ref{equation:dP1-special-log-pair})$ is KLT. In fact, we proved
that
$$
\Bigg(X,\frac{1}{2}\Big(Z_1 + Z_2\Big)\Bigg)
$$
is log canonical (this is only important for
Corollary~\ref{corollary:dP1-lct-lct2-5-3}).
\end{proof}

Arguing as in the~proof of Theorem~\ref{theorem:dP1-lct-lct2}, we
obtain the~following two corollaries.

\begin{corollary}
\label{corollary:dP1-lct-lct2-5-3} If there are no $G$-invariant
curves in $|-K_X|$, then $\mathrm{lct}(X,G)\in\{5/3, 2\}$.
\end{corollary}

\begin{corollary}
\label{corollary:dP1-lct} We have $\mathrm{lct}(X,G)\in\{5/6, 1,
5/3, 2\}$.
\end{corollary}

Using description of the~group $\mathrm{Aut}(X)$ (see
\cite{DoIs06}), we obtain the~following result.

\begin{corollary}
\label{corollary:dP1-main} The following conditions are
equivalent:
\begin{itemize}
\item the~inequality $\mathrm{lct}(X,\mathrm{Aut}(X))>1$ holds,%

\item either $\mathrm{lct}(X,\mathrm{Aut}(X))=5/3$ or $\mathrm{lct}(X,\mathrm{Aut}(X))=2$,%

\item the~pencil $|-K_X|$ does not contain $G$-invariant curves,%

\item the~group $\mathrm{Aut}(X)$ is not Abelian.%
\end{itemize}
\end{corollary}

Let us show how to compute $\mathrm{lct}(X,G)$  in one case.

\begin{lemma}\label{lemma:dP1-D12} If
$f_4(x,y)=x^2y^2$ and $f_6(x,y)=x^6+y^6+x^3y^3$, then
$\mathrm{lct}(X,\mathrm{Aut}(X))=2$.
\end{lemma}

\begin{proof}
Suppose that $f_4(x,y)=x^2y^2$ and $f_6(x,y)=x^6+y^6+x^3y^3$. By
\cite{DoIs06}, we have
$$
\mathrm{Aut}\big(X\big)\cong\mathbb{D}_{6},
$$
and all $\mathrm{Aut}(X)$-invariant curves in $|-2K_{X}|$ can be
described as follows:
\begin{itemize}
\item an irreducible curve that is cut out on $X$ by $z=0$ (see the~proof of Lemma~\ref{lemma:dP1-2K}),%

\item a reducible curve that is cut out on $X$ by $xy=0$,%

\item a reducible curve that is cut out on $X$ by $x^2+y^2=0$,%

\item a reducible curve that is cut out on $X$ by $x^2-y^2=0$.%
\end{itemize}

One can show that $\mathrm{Aut}(X)$-invariant curves in
$|-2K_{X}|$ have at most ordinary double~points, which implies
that $\mathrm{lct}(X,\mathrm{Aut}(X))=2$ by
Theorem~\ref{theorem:dP1-lct-lct2}.
\end{proof}

\section{Double plane ramified in quartic}
\label{sec:degree-two}

Let $X$ be a~smooth quartic surface in $\mathbb{P}(1,1,1,2)$. Then
$X$ can be given by an~equation
$$
t^2=f_4\big(x,y,z\big)\subset\mathbb{P}\big(1,1,1,2\big)\cong\mathrm{Proj}\Big(\mathbb{C}[x,y,z,t]\Big),%
$$
where $\mathrm{wt}(x)=\mathrm{wt}(y)=\mathrm{wt}(z)=1$,
$\mathrm{wt}(t)=2$, and $f_4(x,y,z)$ is a form of degree $4$.

Let $\tau$ be the~involution in $\mathrm{Aut}(X)$ such that
$\tau([x:y:z:t])=[x:y:z:-t]$.

\begin{lemma}[{\cite[Theorem~6.17]{DoIs06}}]
\label{lemma:dP2-Bertini} A $\tau$-invariant subgroup in
$\mathrm{Pic}(X)$ is generated by $-K_{X}$.
\end{lemma}

Let $G$ be a subgroup in $\mathrm{Aut}(X)$ such that $\tau\in G$.
Recall that $\mathrm{Aut}(X)$ is finite.

\begin{lemma}
\label{lemma:dP2-2K} There exists a~$G$-invariant curve in
$|-2K_X|$.
\end{lemma}

\begin{proof}
Let $R$ be the~curve on $X$ that is cut out by $t=0$. Then $R$ is
$G$-invariant.
\end{proof}

\begin{corollary}
\label{corollary:dP2-2K} The inequality
$\mathrm{lct}(X,G)\leqslant 2$ holds.
\end{corollary}

The main purpose of this section is to prove the~following two
results.

\begin{theorem}
\label{theorem:dP2-lct-lct1} Suppose that there exists
a~$G$-invariant curve in $|-K_X|$. Then
$$
\mathrm{lct}\big(X,G\big)=\mathrm{lct}_1\big(X,G\big)\in\big\{3/4,5/6,1\big\}.
$$
\end{theorem}

\begin{proof}
One can easily check that $\mathrm{lct}_1(X,G)\in\{3/4,5/6,1\}$.
It follows from Example~\ref{example:GAFA} that
$$
\mathrm{lct}\big(X,G\big)=\mathrm{lct}_1\big(X,G\big)=\frac{3}{4}
$$
if $\mathrm{lct}_1(X,G)=3/4$. Suppose that
$\mathrm{lct}(X,G)<\mathrm{lct}_1(X,G)$. Let us derive a
contradiction.

There exists a~$G$-invariant effective $\mathbb{Q}$-divisor $D$ on
the~surface $X$ such that
$$
D\sim_{\mathbb{Q}} -K_{X}
$$
and the~log pair $\big(X,\lambda D\big)$ is strictly log canonical
for some rational number $\lambda<\mathrm{lct}_1(X,G)$.

By Theorem~\ref{theorem:connectedness} and
Lemma~\ref{lemma:Nadel-IILC}, the~locus $\mathrm{LCS}(X,\lambda
D)$ consists of a single point $P\in X$.

Let $R$ be the~curve on $X$ that is cut out by $t=0$. Then $P\in
R$, since $\tau\in G$.

Let $L$ be the~unique curve in $|-K_X|$ such that $L$ is singular
at the~point $P$. Then we may assume that $\mathrm{Supp}(D)$ does
not contain any component of the~curve~$L$ by
Remark~\ref{remark:convexity}.~Then
$$
2=L \cdot D\geqslant\mathrm{mult}_P\big(L\big)\cdot \mathrm{mult}_P\big(D\big)\geqslant 2\mathrm{mult}_P\big(D\big)\geqslant \frac{2}{\lambda}>1.%
$$
which is a contradiction.
\end{proof}

\begin{theorem}\label{theorem:dP2-lct-lct2}
Suppose that there are no $G$-invariant curves in $|-K_X|$. Then
$$
1\leqslant\mathrm{lct}\big(X,G\big)=\mathrm{min}\Big(\mathrm{lct}_2\big(X,G\big),\mathrm{lct}_3\big(X,G\big)\Big)\leqslant 2.%
$$
\end{theorem}

\begin{proof}
Arguing as in the~proof of Theorem~\ref{theorem:dP2-lct-lct1} and
using Corollary~\ref{corollary:dP2-2K},~we~have
$$
1\leqslant\mathrm{lct}\big(X,G\big)\leqslant\mathrm{lct}_2\big(X,G\big)\leqslant 2.%
$$

Suppose that $\mathrm{lct}(X,G)<\mathrm{lct}_2(X,G)$ and
$\mathrm{lct}(X,G)<\mathrm{lct}_3(X,G)$. Let us derive a
contradiction.

There exists a~$G$-invariant effective $\mathbb{Q}$-divisor $D$ on
the~surface $X$ such that
$$
D\sim_{\mathbb{Q}} -K_{X}
$$
and  $(X,\lambda D)$ is strictly log canonical for some
$\lambda\in\mathbb{Q}$ such that $\lambda<\mathrm{lct}_2(X,G)$ and
$\lambda<\mathrm{lct}_3(X,G)$.

Let $R$ be the~curve on $X$ that is cut out by $t=0$. It follows
from Lemmata~\ref{lemma:LCS-zero-dimensional} and
\ref{lemma:dP2-Bertini}~that
$$
\mathrm{LCS}\big(X,\lambda D\big)\subset R,
$$
and it follows from Lemma~\ref{lemma:Nadel-IILC} that
$|\mathrm{LCS}(X,\lambda D)|=3$.

Let $P_{1}$, $P_{2}$, $P_{3}$ be three points in $\mathrm{LCS}(X,
\lambda D)$. Then
$$
\mathrm{mult}_{P_{1}}\big(D\big)=\mathrm{mult}_{P_{2}}\big(D\big)=\mathrm{mult}_{P_{3}}\big(D\big)\geqslant\frac{1}{\lambda}>\frac{1}{2}.
$$

Let $\pi\colon X\to\mathbb{P}^{2}$ be a natural projection. Then
$\pi$ is a double cover ramified over the~curve~$\pi(R)$ and
the~points $\pi(P_{1})$, $\pi(P_{2})$, $\pi(P_{3})$ are not
contained in one line by Lemma~\ref{lemma:Nadel-IILC}.

Let $L_{1}$, $L_{2}$, $L_{3}$ be curves in $|-K_{X}|$ such that
$P_{2}\in L_{1}\ni P_{3}$, $P_{1}\in L_{2}\ni P_{3}$, $P_{1}\in
L_{3}\ni P_{2}$.~Then
$$
L_{1}+L_{2}+L_{3}\sim -3K_{X}
$$
and the~divisor $L_{1}+L_{2}+L_{3}$ is $G$-invariant. We may
assume that $\mathrm{Supp}(D)$ does not contain any components of
the~curves $L_{1}$, $L_{2}$, $L_{3}$ by
Remark~\ref{remark:convexity}. Using
\cite[Proposition~8.21]{Ko97}, we see that
\begin{equation}
\label{equation:dP2-special-log-pair}
\Bigg(X, \frac{5}{8}\Big(L_{1}+L_{2}+L_{3}\Big)\Bigg)%
\end{equation}
is log canonical (this is only important for
Corollary~\ref{corollary:dP2-lct-lct2-5-3}). In fact, one can show
that
$$
\Bigg(X, \frac{2}{3}\Big(L_{1}+L_{2}+L_{3}\Big)\Bigg)%
$$
is log canonical $\iff$ $(\ref{equation:dP2-special-log-pair})$ is
KLT. Note that $\pi(L_{1})$, $\pi(L_{2})$, $\pi(L_{3})$ are lines.
We have
$$
6=D\cdot\Big(L_{1}+L_{2}+L_{3}\Big)\geqslant 2\sum_{i=1}^{3}\mathrm{mult}_{P_{i}}\big(D\big)=6\mathrm{mult}_{P_{1}}\big(D\big)=6\mathrm{mult}_{P_{2}}\big(D\big)=6\mathrm{mult}_{P_{3}}\big(D\big),%
$$
which implies that
$\mathrm{mult}_{P_{1}}(D)=\mathrm{mult}_{P_{2}}(D)=\mathrm{mult}_{P_{3}}(D)\leqslant
1$.

Let $T_{1}$, $T_{2}$, $T_{3}$ be the~curves in $|-K_{X}|$ that are
singular at $P_{1}$, $P_{2}$, $P_{3}$, respectively.~Then
$$
T_{1}+T_{2}+T_{3}\sim -3K_{X}
$$
and the~divisor $T_{1}+T_{2}+T_{3}$ is $G$-invariant. We may
assume that $\mathrm{Supp}(D)$ does not contain any components of
the~curves $T_{1}$, $T_{2}$, $T_{3}$ by
Remark~\ref{remark:convexity}. Using
\cite[Proposition~8.21]{Ko97}, we see that
\begin{equation}
\label{equation:dP2-special-log-pair-tangent}
\Bigg(X, \frac{5}{8}\Big(T_{1}+T_{2}+T_{3}\Big)\Bigg)%
\end{equation}
is log canonical (this is only important for
Corollary~\ref{corollary:dP2-lct-lct2-5-3}). Moreover, one can
show that
$$
\Bigg(X, \frac{2}{3}\Big(T_{1}+T_{2}+T_{3}\Big)\Bigg)%
$$
is log canonical $\iff$
$(\ref{equation:dP2-special-log-pair-tangent})$ is KLT $\iff$
$T_{1}+T_{2}+T_{3}\ne L_{1}+L_{2}+L_{3}$.

Note that $\pi(T_{1})$, $\pi(T_{2})$, $\pi(T_{3})$ are lines
tangent to $\pi(R)$ at  $\pi(P_{1})$, $\pi(P_{2})$,
$\pi(P_{3})$,~respectively.

If $T_{1}+T_{2}+T_{3}=L_{1}+L_{2}+L_{3}$, then
$\mathrm{mult}_{P_{1}}(D)=\mathrm{mult}_{P_{2}}(D)=\mathrm{mult}_{P_{3}}(D)\leqslant
2/3$, since
$$
6=D\cdot\Big(L_{1}+L_{2}+L_{3}\Big)\geqslant 3\sum_{i=1}^{3}\mathrm{mult}_{P_{i}}\big(D\big)=9\mathrm{mult}_{P_{1}}\big(D\big)=9\mathrm{mult}_{P_{2}}\big(D\big)=9\mathrm{mult}_{P_{3}}\big(D\big).%
$$

Let $Z_{1}$, $Z_{2}$ and $Z_{3}$ be a curves in $|-2K_{X}|$ such
that $\pi(Z_{1})$, $\pi(Z_{2})$, $\pi(Z_{3})$ are conics where
$$
\Big\{\pi\big(P_{1}\big),\pi\big(P_{2}\big),\pi\big(P_{3}\big)\Big\}\subset\pi\big(Z_{1}\big)\cap\pi\big(Z_{2}\big)\cap\pi\big(Z_{3}\big),%
$$
the~conic $\pi(Z_{1})$ is tangent to $\pi(R)$ at $\pi(P_{2})$ and
$\pi(P_{3})$, the~conic $\pi(Z_{2})$ is tangent to  $\pi(R)$ at
the~points $\pi(P_{1})$ and $\pi(P_{3})$, and $\pi(Z_{3})$ is
tangent to  $\pi(R)$ at  $\pi(P_{1})$ and $\pi(P_{2})$. Then
$$
Z_{1}+Z_{2}+Z_{3}=2\Big(T_{1}+T_{2}+T_{3}\Big)\iff T_{1}+T_{2}+T_{3}=L_{1}+L_{2}+L_{3}%
$$
and the~conics  $\pi(Z_{1})$, $\pi(Z_{2})$, $\pi(Z_{3})$ are
irreducible $\iff$ $T_{1}+T_{2}+T_{3}\ne L_{1}+L_{2}+L_{3}$. Then
$$
\Bigg(X, \frac{1}{3}\Big(Z_{1}+Z_{2}+Z_{3}\Big)\Bigg)%
$$
is log canonical  if $T_{1}+T_{2}+T_{3}\ne L_{1}+L_{2}+L_{3}$ (see
Example~\ref{example:three-conics} and
\cite[Proposition~8.21]{Ko97}).~However
$$
Z_{1}+Z_{2}+Z_{3}\sim -6K_{X}
$$
and the~divisor $Z_{1}+Z_{2}+Z_{3}$ is $G$-invariant. Thus, we may
assume that $\mathrm{Supp}(D)$ does not~contain any components of
the~curves $Z_{1}$, $Z_{2}$, $Z_{3}$ by
Remark~\ref{remark:convexity}. Then
$$
12=D\cdot\Big(Z_{1}+Z_{2}+Z_{3}\Big)\geqslant 5\sum_{i=1}^{3}\mathrm{mult}_{P_{i}}\big(D\big)=15\mathrm{mult}_{P_{1}}\big(D\big)=15\mathrm{mult}_{P_{2}}\big(D\big)=15\mathrm{mult}_{P_{3}}\big(D\big),%
$$
which implies that
$\mathrm{mult}_{P_{1}}(D)=\mathrm{mult}_{P_{2}}(D)=\mathrm{mult}_{P_{3}}(D)\leqslant
4/5$. If $Z_{1}=Z_{2}=Z_{3}$, then
$$
4=D\cdot Z_{1}=D\cdot Z_{2}=D\cdot Z_{3}\geqslant 2\sum_{i=1}^{3}\mathrm{mult}_{P_{i}}\big(D\big)=6\mathrm{mult}_{P_{1}}\big(D\big)=6\mathrm{mult}_{P_{2}}\big(D\big)=6\mathrm{mult}_{P_{3}}\big(D\big),%
$$
which implies that
$\mathrm{mult}_{P_{1}}(D)=\mathrm{mult}_{P_{2}}(D)=\mathrm{mult}_{P_{3}}(D)\leqslant
2/3$.

Let $\sigma\colon\bar{X}\to X$ be the~blow-up of the~surface $X$
at  $P_1$, $P_2$ and $P_{3}$, let $E_{1}$, $E_{2}$ and $E_{3}$ be
the~exceptional curves of the~blow up $\sigma$ such that
$\sigma(E_{1})=P_{1}$, $\sigma(E_{2})=P_{2}$ and
$\sigma(E_{3})=P_{3}$.~Then
$$
K_{\bar{X}}+\lambda\bar{D}+\sum_{i=1}^{3}\Big(\lambda\mathrm{mult}_{P_i}\big(D\big)-1\Big)E_i\sim_{\mathbb{Q}}\sigma^{*}\Big(K_X + \lambda D\Big),%
$$
where $\bar{D}$ is the~proper transform of the~divisor $D$ on
the~surface $\bar{X}$.

It follows from Remark~\ref{remark:blow-up-inequality} that there
are points $Q_{1}\in E_{1}$, $Q_{2}\in E_{2}$ and $Q_{3}\in E_{3}$
such that
$$
\mathrm{LCS}\Bigg(\bar{X},\lambda\bar{D}+\sum_{i=1}^{3}\Big(\lambda\mathrm{mult}_{P_i}\big(D\big)-1\Big)E_i\Bigg)=\Big\{Q_{1},Q_{2},Q_{3}\Big\},%
$$
as
$\lambda\mathrm{mult}_{P_1}(D)-1=\lambda\mathrm{mult}_{P_2}(D)-1=\lambda\mathrm{mult}_{P_3}(D)-1<1$.
By Remark~\ref{remark:blow-up-inequality}, we have
\begin{equation}
\label{equation:dP2-inequality-I}
\mathrm{mult}_{P_{1}}\big(D\big)+\mathrm{mult}_{Q_{1}}\big(\bar{D}\big)=\mathrm{mult}_{P_{2}}\big(D\big)+\mathrm{mult}_{Q_{2}}\big(\bar{D}\big)=\mathrm{mult}_{P_{3}}\big(D\big)+\mathrm{mult}_{Q_{3}}\big(\bar{D}\big)\geqslant\frac{2}{\lambda}>1,
\end{equation}
where
$\mathrm{mult}_{Q_{1}}(\bar{D})=\mathrm{mult}_{Q_{2}}(\bar{D})=\mathrm{mult}_{Q_{3}}(\bar{D})$,
since the~divisor $D$ is $G$-invariant.

Note that the~action of the~group $G$ on the~surface $X$ naturally
lifts to an action on $\bar{X}$.

Since the line $\pi(L_{1})$ is not tangent to $\pi(R)$ at both
$\pi(P_{2})$ and $\pi(P_{3})$, without loss of generality, we~may
assume that $\pi(L_{1})$ intersects transversally $\pi(R)$ at
$\pi(P_{2})$. Similarly, we may assume~that
\begin{itemize}
\item the~line $\pi(L_{2})$ intersects transversally the~curve $\pi(R)$ at the~point $\pi(P_{3})$,%
\item the~line $\pi(L_{3})$ intersects transversally the~curve $\pi(R)$ at the~point $\pi(P_{1})$.%
\end{itemize}

Let $\bar{L}_{1}$, $\bar{L}_{2}$, $\bar{L}_{3}$ be the~proper
transforms of the~curves $L_{1}$, $L_{2}$, $L_{3}$ on the~surface
$\bar{X}$,~respectively.~Then
$$
2-\sum_{i=2}^{3}\mathrm{mult}_{P_{i}}\big(L_{1}\big)\mathrm{mult}_{P_{i}}\big(D\big)=\bar{L}_{1}\cdot\bar{D}\geqslant\sum_{i=2}^{3}\mathrm{mult}_{Q_{i}}\big(\bar{L}_{1}\big)\mathrm{mult}_{Q_{i}}\big(\bar{D}\big),%
$$
which implies that $Q_{2}\not\in\bar{L}_{1}$ by
$(\ref{equation:dP2-inequality-I})$. Similarly, we see that
$Q_{3}\not\in\bar{L}_{2}$ and $Q_{1}\not\in\bar{L}_{3}$.

Let $\bar{R}$ be the~proper transform of the~curve $R$ on
the~surface $\bar{X}$. Then
$$
Q_{1}=\bar{R}\cap E_{1}
$$
since the~$\sigma$-exceptional curve $E_{1}$ contains exactly
two points that are fixed by the~involution~$\tau$, which are
$\bar{R}\cap E_{1}$ and $\bar{L}_{3}\cap E_{1}$. Similarly, we see
that $Q_{2}=\bar{R}\cap E_{2}$ and $Q_{3}=\bar{R}\cap E_{3}$.

By Remark~\ref{remark:convexity}, we may assume that
$\bar{R}\not\subseteq\mathrm{Supp}(\bar{D})$, since $R$ is smooth.
Then
$$
\sum_{i=1}^{3}\mathrm{mult}_{Q_i}\big(\bar{D}\big)\leqslant\bar{D}\cdot\bar{R}=4-\sum_{i=1}^{3}\mathrm{mult}_{P_i}\big(D\big),%
$$
where
$\mathrm{mult}_{Q_1}(\bar{D})+\mathrm{mult}_{P_1}(D)=\mathrm{mult}_{Q_2}(\bar{D})+\mathrm{mult}_{P_2}(D)=\mathrm{mult}_{Q_3}(\bar{D})+\mathrm{mult}_{P_3}(D)$.
Then
\begin{equation}
\label{equation:dP2-4-3}
\mathrm{mult}_{Q_1}\big(\bar{D}\big)+\mathrm{mult}_{P_1}\big(D\big)=\mathrm{mult}_{Q_2}\big(\bar{D}\big)+\mathrm{mult}_{P_2}\big(D\big)=\mathrm{mult}_{Q_3}\big(\bar{D}\big)+\mathrm{mult}_{P_3}\big(D\big)\leqslant\frac{4}{3}.%
\end{equation}

Let  $\rho\colon\tilde{X} \to \bar{X}$ be a blow up of the~surface
$\bar{X}$ at the~points $Q_{1}$, $Q_{2}$, $Q_{3}$ and let $F_{1}$,
$F_{2}$ and $F_{3}$~be the~exceptional curves of the~blow up
$\rho$ such that $\rho(F_{1})=Q_{1}$, $\rho(F_{2})=Q_{2}$ and
$\rho(F_{2})=Q_{3}$.~Then
$$
K_{\tilde{X}}+\lambda\tilde{D}+\sum_{i=1}^3\Big(\lambda \mathrm{mult}_{P_i}\big(D\big)-1\Big)\tilde{E}_i+\sum_{i=1}^3\Big(\lambda \mathrm{mult}_{Q_i}\big(\bar{D}\big)+\lambda\mathrm{mult}_{P_i}\big(D\big)-2\Big)F_i \sim_{\mathbb{Q}}\big(\sigma\circ\rho\big)^{*}\Big(K_{X}+\lambda D\Big),%
$$
where $\tilde{D}$ and $\tilde{E}_{i}$ are proper transforms of
the~divisors $D$ and $E_{i}$ on the~surface $\tilde{X}$,
respectively.

It follows from Remark~\ref{remark:blow-up-inequality} that there
are points $O_{1}\in F_{1}$, $O_{2}\in F_{2}$ and $O_{3}\in F_{3}$
such that
$$
\mathrm{LCS}\Bigg(\tilde{X}, \lambda\tilde{D}+\sum_{i=1}^3\Big(\lambda \mathrm{mult}_{P_i}\big(D\big)-1\Big)\tilde{E}_i+\sum_{i=1}^3\Big(\lambda \mathrm{mult}_{Q_i}\big(\bar{D}\big)+\lambda\mathrm{mult}_{P_i}\big(D\big)-2\Big)F_i\Bigg)=\Big\{O_{1},O_{2},O_{3}\Big\},%
$$
since
$\mathrm{mult}_{Q_1}(\bar{D})+\mathrm{mult}_{P_1}(D)=\mathrm{mult}_{Q_2}(\bar{D})+\mathrm{mult}_{P_2}(D)=\mathrm{mult}_{Q_3}(\bar{D})+\mathrm{mult}_{P_3}(D)\leqslant
4/3$.

The~action of the~group $G$ on the~surface $\bar{X}$ naturally
lifts to an action on the surface $\tilde{X}$ such that every
curve among the curves $F_{1}$, $F_{2}$ and $F_{3}$ contain
exactly two $\tau$-fixed points.

Let $\tilde{R}$ be the~proper transform of the~curve $R$ on
the~surface $\tilde{X}$. Then
\begin{itemize}
\item either $O_{1}=\tilde{E}_{1}\cap F_{1}$, $O_{2}=\tilde{E}_{2}\cap F_{2}$ and $O_{3}=\tilde{E}_{3}\cap F_{3}$,%
\item or $O_{1}=\tilde{R}\cap F_{1}$, $O_{2}=\tilde{R}\cap F_{2}$ and $O_{3}=\tilde{R}\cap F_{3}$.%
\end{itemize}

Suppose that $O_{1}=\tilde{R}\cap F_{1}$, $O_{2}=\tilde{R}\cap
F_{2}$ and $O_{3}=\tilde{R}\cap F_{3}$. Then
$$
\mathrm{LCS}\Bigg(\tilde{X}, \lambda\tilde{D}+\sum_{i=1}^3\Big(\lambda \mathrm{mult}_{Q_i}\big(\bar{D}\big)+\lambda\mathrm{mult}_{P_i}\big(D\big)-2\Big)F_i\Bigg)=\Big\{O_{1},O_{2},O_{3}\Big\},%
$$
since $O_{1}\not\in\tilde{E}_{1}$, $O_{2}\not\in\tilde{E}_{2}$
and $O_{3}\not\in\tilde{E}_{3}$. Then it follows from
Remark~\ref{remark:blow-up-inequality} that
\begin{equation}
\label{equation:dP2-inequality-V}
\mathrm{mult}_{O_{i}}\big(\tilde{D}\big)+\mathrm{mult}_{Q_i}\big(\bar{D}\big)+\mathrm{mult}_{P_i}\big(D\big)\geqslant\frac{3}{\lambda}>\frac{3}{2}%
\end{equation}
for every $i\in\{1,2,3\}$, where
$\mathrm{mult}_{O_{1}}(\tilde{D})=\mathrm{mult}_{O_{2}}(\tilde{D})=\mathrm{mult}_{O_{2}}(\tilde{D})$.
However
$$
4-\sum_{i=1}^{3}\mathrm{mult}_{P_i}\big(D\big)+\sum_{i=1}^{3}\mathrm{mult}_{Q_i}\big(\bar{D}\big)=\tilde{R}\cdot\tilde{D}\geqslant\sum_{i=1}^{3}\mathrm{mult}_{O_i}\big(\tilde{D}\big),%
$$
which contradicts $(\ref{equation:dP2-inequality-V})$. Thus, we
see that $O_{1}=\tilde{E}_{1}\cap F_{1}$, $O_{2}=\tilde{E}_{2}\cap
F_{2}$ and $O_{3}=\tilde{E}_{3}\cap F_{3}$.

If
$6(\lambda\mathrm{mult}_{P_1}(D)-1)+(\lambda\mathrm{mult}_{Q_1}(\bar{D})+\lambda\mathrm{mult}_{P_1}(D)-2)<4$,
then we can apply Corollary~\ref{corollary:Cheltsov-Kosta}~to
$$
\Bigg(\tilde{X}, \lambda\tilde{D}+\Big(\lambda \mathrm{mult}_{P_1}\big(D\big)-1\Big)\tilde{E}_1+\Big(\lambda \mathrm{mult}_{Q_1}\big(\bar{D}\big)+\lambda\mathrm{mult}_{P_1}\big(D\big)-2\Big)F_1\Bigg),%
$$
which immediately gives a contradiction, because
$$
\lambda\tilde{D}\cdot F_{1}=\lambda\mathrm{mult}_{Q_1}\big(\bar{D}\big)\leqslant 1+\frac{3}{2}\Big(\lambda\mathrm{mult}_{Q_1}\big(\bar{D}\big)+\lambda\mathrm{mult}_{P_1}\big(D\big)-2\Big)-\Big(\lambda\mathrm{mult}_{P_1}\big(D\big)-1\Big)%
$$
and $\lambda\tilde{D}\cdot
\tilde{E}_{1}=2(\lambda\mathrm{mult}_{P_1}(D)-1)-(\lambda\mathrm{mult}_{Q_1}(\bar{D})+\lambda\mathrm{mult}_{P_1}(D)-2)$.
Hence
$$
6\Big(\lambda\mathrm{mult}_{P_1}\big(D\big)-1\Big)+\Big(\lambda\mathrm{mult}_{Q_1}\big(\bar{D}\big)+\lambda\mathrm{mult}_{P_1}\big(D\big)-2\Big)\geqslant 4,%
$$
which implies that
$7\mathrm{mult}_{P_1}(D)+\mathrm{mult}_{Q_1}(\bar{D})\geqslant
12/\lambda$. Similarly,
\begin{equation}
\label{equation:dP2-7-1-6}
7\mathrm{mult}_{P_1}\big(D\big)+\mathrm{mult}_{Q_1}\big(\bar{D}\big)=7\mathrm{mult}_{P_2}\big(D\big)+\mathrm{mult}_{Q_2}\big(\bar{D}\big)=7\mathrm{mult}_{P_3}\big(D\big)+\mathrm{mult}_{Q_3}\big(\bar{D}\big)\geqslant \frac{12}{\lambda},%
\end{equation}
which implies that
$\mathrm{mult}_{P_1}(D)=\mathrm{mult}_{P_2}(D)=\mathrm{mult}_{P_3}(D)>7/9$
by $(\ref{equation:dP2-4-3})$. Then
$$
T_{1}+T_{2}+T_{3}\ne L_{1}+L_{2}+L_{3},
$$
since
$\mathrm{mult}_{P_{1}}(D)=\mathrm{mult}_{P_{2}}(D)=\mathrm{mult}_{P_{3}}(D)\leqslant
2/3$ if $T_{1}+T_{2}+T_{3}=L_{1}+L_{2}+L_{3}$. We have
$$
Z_{1}\ne Z_{2}\ne Z_{3}\ne Z_{1},
$$
since
$\mathrm{mult}_{P_{1}}(D)=\mathrm{mult}_{P_{2}}(D)=\mathrm{mult}_{P_{3}}(D)\leqslant
2/3$ if $Z_{1}=Z_{2}=Z_{3}$.

Let $\mathcal{M}$ be linear subsystem in $|-3K_{X}|$ such that
$M\in\mathcal{M}$ if $\pi(M)$ is a cubic curve such that
$$
\Big\{\pi\big(P_{1}\big),\pi\big(P_{2}\big),\pi\big(P_{3}\big)\Big\}\subset\pi\big(M\big)%
$$
and $\pi(M)$ is tangent to $\pi(R)$ at the points $\pi(P_{1})$,
$\pi(P_{2})$ and $\pi(P_{3})$. Then
$$
T_{1}+T_{2}+T_{3}\in\mathcal{M}\ni L_{1}+L_{2}+L_{3}
$$
and every curve in $\mathcal{M}$ is singular at the points
$P_{1}$, $P_{2}$ and $P_{3}$. Note that
$\mathrm{dim}(\mathcal{M})\geqslant 3$.

Let $\bar{\mathcal{M}}$ be the proper transform of the linear
system $\mathcal{M}$ on the surface $\bar{X}$. Then
$$
\bar{\mathcal{M}}\sim\sigma^{*}\Big(-3K_{X}\Big)-\sum_{i=1}^{3}\mathrm{mult}_{P_{i}}\big(\mathcal{M}\big)E_{3},%
$$
where
$\mathrm{mult}_{P_{1}}(\mathcal{M})=\mathrm{mult}_{P_{2}}(\mathcal{M})=\mathrm{mult}_{P_{3}}(\mathcal{M})\geqslant
2$.

Let $\bar{\mathcal{B}}$ be a linear subsystem of the linear system
$\bar{\mathcal{M}}$ consisting of curves that pass through the
points $Q_{1}$,  $Q_{2}$ and $Q_{3}$. Then
$\bar{\mathcal{B}}\ne\varnothing$, since
$\mathrm{dim}(\mathcal{M})\geqslant 3$. Put
$\mathcal{B}=\sigma(\bar{\mathcal{B}})$. Then
$$
\bar{\mathcal{B}}\sim\sigma^{*}\Big(-3K_{X}\Big)-\sum_{i=1}^{3}\mathrm{mult}_{P_{i}}\big(\mathcal{B}\big)E_{3},%
$$
where
$\mathrm{mult}_{P_{1}}(\mathcal{B})=\mathrm{mult}_{P_{2}}(\mathcal{B})=\mathrm{mult}_{P_{3}}(\mathcal{B})\geqslant
\mathrm{mult}_{P_{1}}(\mathcal{M})=\mathrm{mult}_{P_{2}}(\mathcal{M})=\mathrm{mult}_{P_{3}}(\mathcal{M})\geqslant
2$.

Note that the linear systems $\mathcal{M}$, $\bar{\mathcal{B}}$
and $\mathcal{B}$ are $G$-invariant.

Let $B$ be a general curve in the linear system $\mathcal{B}$.
Since $|-K_{X}|$  contains no $G$-invariant curves, we see that
either $\mathcal{B}=B$ or $\mathcal{B}$ has no fixed curves. If
$\mathcal{B}=B$, then $B$ is $G$-invariant and
\begin{equation}
\label{equation:dP2-spacial-log-pair}
\Bigg(X, \frac{\lambda}{3} B\Bigg)%
\end{equation}
is log canonical. Indeed, if the log pair
$(\ref{equation:dP2-spacial-log-pair})$ is not log canonical, then
$$
3>\mathrm{mult}_{P_{1}}\big(B\big)>\frac{7}{3}>2,
$$
because we can apply the arguments we used for $(X,\lambda D)$ to
the log pair $(\ref{equation:dP2-spacial-log-pair})$.

We may assume that $B$ is not contained in $\mathrm{Supp}(D)$ by
Remark~\ref{remark:convexity}. Then
$$
6=B\cdot
D\geqslant\sum_{i=1}^{3}\mathrm{mult}_{P_{i}}\big(B\big)\mathrm{mult}_{P_{i}}\big(D\big)>\frac{7}{9}\sum_{i=1}^{3}\mathrm{mult}_{P_{i}}\big(B\big)=\frac{7}{3}\mathrm{mult}_{P_{1}}\big(B\big)=\frac{7}{3}\mathrm{mult}_{P_{2}}\big(B\big)=\frac{7}{3}\mathrm{mult}_{P_{1}}\big(B\big),%
$$
which implies that
$\mathrm{mult}_{P_{1}}(B)=\mathrm{mult}_{P_{2}}(B)=\mathrm{mult}_{P_{3}}(B)=2$.

Let $\bar{B}$ be the proper transform of the curve $B$ on the
surface $\bar{X}$. Then $\bar{B}\in\bar{\mathcal{B}}$ and
$$
6-6\mathrm{mult}_{P_{1}}\big(D\big)=\bar{B}\cdot\bar{D}\geqslant\sum_{i=1}^{3}\mathrm{mult}_{Q_{i}}\big(\bar{B}\big)\mathrm{mult}_{Q_{i}}\big(\bar{D}\big)\geqslant3\mathrm{mult}_{Q_{1}}\big(\bar{D}\big)=3\mathrm{mult}_{Q_{2}}\big(\bar{D}\big)=3\mathrm{mult}_{Q_{3}}\big(\bar{D}\big),%
$$
which implies that
$2\mathrm{mult}_{P_{1}}(D)+\mathrm{mult}_{Q_{1}}(\bar{D})\leqslant
2$. By $(\ref{equation:dP2-7-1-6})$, we have
$$
5\mathrm{mult}_{P_1}\big(D\big)+2\geqslant 7\mathrm{mult}_{P_1}\big(D\big)+\mathrm{mult}_{Q_1}\big(\bar{D}\big)\geqslant\frac{12}{\lambda}>6,%
$$
which implies that $\mathrm{mult}_{P_1}(D)>4/5$. But we already
proved that $\mathrm{mult}_{P_1}(D)\leqslant 4/5$.
\end{proof}

Arguing as in the~proof of Theorem~\ref{theorem:dP2-lct-lct2}, we
obtain the~following two corollaries.

\begin{corollary}
\label{corollary:dP2-lct-lct2-5-3} If there are no $G$-invariant
curves in $|-K_X|$, then $\mathrm{lct}(X,G)\in\{15/8, 2\}$.
\end{corollary}

\begin{corollary}
\label{corollary:dP2-lct-lct2-orbits} The equality
$\mathrm{lct}(X,G)=2$ holds if the~following two conditions are
satisfied:
\begin{itemize}
\item the~linear system $|-K_X|$ does not contain $G$-invariant curves,%
\item the~surface $X$ does not have $G$-orbits of length $3$.
\end{itemize}
\end{corollary}

\begin{corollary}
\label{corollary:dP2-lct} We have $\mathrm{lct}(X,G)\in\{3/4, 5/6,
1, 15/8, 2\}$.
\end{corollary}

Using description of the~group $\mathrm{Aut}(X)$ (see
\cite{DoIs06}), we obtain the~following result.

\begin{corollary}
\label{corollary:dP2-main} The following conditions are
equivalent:
\begin{itemize}
\item the~inequality $\mathrm{lct}(X,\mathrm{Aut}(X))>1$ holds,%

\item the~equality $\mathrm{lct}(X,\mathrm{Aut}(X))=2$ holds,%

\item the~linear system $|-K_X|$ does not contain $\mathrm{Aut}(X)$-invariant curves,%

\item the~group $\mathrm{Aut}(X)$ is isomorphic to one of
the~following groups:
$$
\mathbb{S}_4\times\mathbb{Z}_2, \big(\mathbb{Z}_4^2\rtimes\mathbb{S}_3\big)\times\mathbb{Z}_2, \mathbb{PSL}\big(2,\mathbb{F}_7\big)\times\mathbb{Z}_2.%
$$%
\end{itemize}
\end{corollary}

Let us show how to compute $\mathrm{lct}(X,G)$  in two cases.

\begin{lemma}\label{lemma:dP2-Klein}
Suppose that $f_4(x,y,z)=x^3y+y^3z+z^3x$ and $G\cong
\mathbb{Z}_{2}\times (\mathbb{Z}_{7}\rtimes\mathbb{Z}_{3})$. Then
$$
\mathrm{lct}\big(X,G\big)=\mathrm{lct}_{3}\big(X,G\big)=\frac{15}{8}<\mathrm{lct}_{2}\big(X,G\big)=2.
$$
\end{lemma}

\begin{proof}
One can easily check that the~linear system $|-K_{X}|$ does not
contain $G$-invariant curves, and the~only $G$-invariant curve in
$|-2K_{X}|$ is a curve that is cut out on $X$ by $t=0$. Then
$$
2=\mathrm{lct}_{2}\big(X,G\big)\geqslant \mathrm{lct}\big(X,G\big)=\mathrm{min}\Big(2,\mathrm{lct}_{3}\big(X,G\big)\Big)\in\Big\{2, 15/8\Big\}%
$$
by Theorem~\ref{theorem:dP2-lct-lct2} and
Corollary~\ref{corollary:dP2-lct-lct2-5-3}. Note that
$\mathrm{Aut}(X)\cong\mathbb{Z}_{2}\times
\mathbb{PSL}(2,\mathbb{F}_{7})$.

Put $P_{1}=[1:0:0:0]$, $P_{2}=[0:1:0:0]$, $P_{3}=[0:0:1:0]$. Then
\begin{itemize}
\item the points $P_{1}$, $P_{2}$, $P_{3}$ are contained in the unique $\mathrm{Aut}(X)$-orbit consisting of $24$ points,%
\item the stabilizer subgroup of the subset $\{P_{1}, P_{2}, P_{3}\}$ is isomorphic to $\mathbb{Z}_{2}\times (\mathbb{Z}_{7}\rtimes\mathbb{Z}_{3})$.%
\end{itemize}

Without loss of generality, we may assume that  $\{P_{1}, P_{2},
P_{3}\}$~is~$G$-invariant.

The linear system  $|-K_{X}|$ contains curves $C_{1}$, $C_{2}$ and
$C_{3}$ such that
$$
\mathrm{mult}_{P_{1}}\big(C_{1}\big)=\mathrm{mult}_{P_{2}}\big(C_{2}\big)=\mathrm{mult}_{P_{3}}\big(C_{3}\big)=2,
$$
and the~curves $C_{1}$, $C_{2}$ and $C_{3}$ have cusps at
the~points $P_{1}$, $P_{2}$ and $P_{3}$, respectively. Then
$$
\Bigg(X, \frac{5}{8}\Big(C_{1}+C_{2}+C_{3}\Big)\Bigg)
$$
is strictly log canonical, which implies that
$\mathrm{lct}_{3}(X,G)\leqslant 15/8$.
\end{proof}

\begin{lemma}\label{lemma:dP2-Z2-Z2-Z2} Suppose that
$$
f_4\big(x,y,z\big)=t^2+z^4+y^4+x^4+ax^2y^2+bx^2z^2+cy^2z^2,
$$
where $a$, $b$ and $c$ are general complex numbers. Then
$\mathrm{lct}(X,\mathrm{Aut}(X))=1$.
\end{lemma}

\begin{proof}
It follows from \cite{DoIs06} that
$$
\mathrm{Aut}\big(X\big)\cong\mathbb{Z}_2\times\mathbb{Z}_2\times\mathbb{Z}_2,%
$$
which implies that every $\mathrm{Aut}(X)$-invariant curve in
$|-K_{X}|$ is cut out on $X$ by one of the~following equations:
$x=0$, $y=0$, $z=0$. Then $\mathrm{lct}(X,\mathrm{Aut}(X))=1$ by
Theorem~\ref{theorem:dP2-lct-lct1}.
\end{proof}

\section{Cubic surfaces}
\label{sec:degree-three}

Let $X$ be a~smooth cubic surface in $\mathbb{P}^3$. Then
$\mathrm{Aut}(X)$ is finite. It follows from \cite{DoIs06} that
\begin{itemize}
\item if $\mathrm{Aut}(X)\cong\mathbb{S}_5$, then $X$ is the Clebsch cubic surface,%
\item if $\mathrm{Aut}(X)\cong\mathbb{Z}_3^2 \rtimes \mathbb{S}_4$, then $X$ is the Fermat cubic surface.%
\end{itemize}

\begin{lemma}[{\cite[Example~1.11]{Ch07b}}]
\label{lemma:dP3-Clebsch} If $\mathrm{Aut}(X)\cong\mathbb{S}_5$,
then $\mathrm{lct}(X,\mathrm{Aut}(X))=2$.
\end{lemma}

\begin{lemma}[{\cite[Lemma~5.6]{Ch07b}}]
\label{lemma:dP3-Fermat} If $\mathrm{Aut}(X)\cong\mathbb{Z}_3^2
\rtimes \mathbb{S}_4$, then $\mathrm{lct}(X,\mathrm{Aut}(X))=4$.
\end{lemma}

By \cite{DoIs06}, there is a~$\mathrm{Aut}(X)$-invariant curve in
$|-K_X|$ if $\mathrm{Aut}(X)\not\cong\mathbb{S}_5$ and
$\mathrm{Aut}(X)\not\cong\mathbb{Z}_3^2 \rtimes \mathbb{S}_4$.

\begin{corollary}
\label{corollary:dP3-K} If $\mathrm{Aut}(X)\not\cong\mathbb{S}_5$
and $\mathrm{Aut}(X)\not\cong\mathbb{Z}_3^2 \rtimes \mathbb{S}_4$,
then $\mathrm{lct}(X,\mathrm{Aut}(X))\leqslant 1$.
\end{corollary}

The main purpose of this section is to prove the following result.

\begin{theorem}
\label{theorem:dP3-lct-lct1} Let $G$ be a subgroup of the group
$\mathrm{Aut}(X)$. Then
$$
\mathrm{lct}\big(X,G\big)=\mathrm{lct}_1\big(X,G\big)\in\big\{2/3,5/6,1\big\}
$$
if the following two conditions are satisfied:
\begin{itemize}
\item the linear system $|-K_{X}|$ contains a~$G$-invariant curve,%
\item a $G$-invariant subgroup in $\mathrm{Pic}(X)$ is generated by $-K_{X}$.%
\end{itemize}
\end{theorem}

\begin{proof}
Suppose $|-K_{X}|$ contain a~$G$-invariant curve. Then
$$
\mathrm{lct}_1\big(X,G\big)\in\big\{2/3,3/4, 5/6,1\big\},
$$
and it follows from Example~\ref{example:GAFA} that
$\mathrm{lct}(X,G)=\mathrm{lct}_1(X,G)=2/3$ if
$\mathrm{lct}_1(X,G)=2/3$.

Suppose that a $G$-invariant subgroup in
$\mathrm{Pic}(X)$~is~$\mathbb{Z}[-K_{X}]$. Then
$\mathrm{lct}_1(X,G)\ne 4/3$.

Suppose that $\mathrm{lct}(X,G)<\mathrm{lct}_1(X,G)\ne 2/3$. Let
us derive a contradiction.

There exists a~$G$-invariant effective $\mathbb{Q}$-divisor $D$ on
the~surface $X$ such that
$$
D\sim_{\mathbb{Q}} -K_{X}
$$
and the~log pair $\big(X,\lambda D\big)$ is strictly log canonical
for some rational number $\lambda<\mathrm{lct}_1(X,G)$.

By Theorem~\ref{theorem:connectedness} and
Lemma~\ref{lemma:Nadel-IILC}, the~locus $\mathrm{LCS}(X,\lambda
D)$ consists of a single point $P\in X$.

Let $T$ be the~curve in $|-K_X|$ such that
$\mathrm{mult}_{P}(T)\geqslant 2$. We may assume that
$\mathrm{Supp}(D)$ does not contain any component of the~curve~$T$
by Remark~\ref{remark:convexity}.~Then
$$
3=T \cdot D\geqslant\mathrm{mult}_P\big(T\big)\cdot \mathrm{mult}_P\big(D\big)\geqslant 2\mathrm{mult}_P\big(D\big)\geqslant \frac{2}{\lambda}>1,%
$$
which implies $\mathrm{mult}_P(T)=2$ and
$\mathrm{mult}_P(D)\leqslant 3/2$.

Note that the curve $T$ is irreducible, which implies that
$P=\mathrm{Sing}(T)$.

Let $\sigma\colon\bar{X}\to X$ be a blow up of the point $P$ and
let $E$ be the $\sigma$-exceptional curve. Then
$$
K_{\bar{X}}+\lambda\bar{D}+\Big(\lambda\mathrm{mult}_{P}\big(D\big)-1\Big)E\sim_{\mathbb{Q}}\sigma^{*}\Big(K_{X}+\lambda D\Big),%
$$
where $\bar{D}$ is the proper transform of the divisor $D$ on the
surface $\bar{X}$.

It follows from  Remark~\ref{remark:blow-up-inequality} that there
exists a point $Q\in E$~such~that
$$
\mathrm{LCS}\Bigg(\bar{X},\ \lambda\bar{D}+\Big(\lambda\mathrm{mult}_{P}\big(D\big)-1\Big)E\Bigg)=Q%
$$
and $\mathrm{mult}_{Q}(\bar{D})+\mathrm{mult}_{P}(D)\geqslant
2/\lambda$.

Let $\bar{T}$ be the proper transform of the curve $T$ on the
surface $\bar{X}$. If $Q\in\bar{T}$, then
$$
3-2\mathrm{mult}_{P}\big(D\big)=\bar{T}\cdot\bar{D}\geqslant\mathrm{mult}_{Q}\big(\bar{T}\big)\mathrm{mult}_{Q}\big(\bar{D}\big)>\mathrm{mult}_{Q}\big(\bar{T}\big)\Big(2-\mathrm{mult}_{P}\big(D\big)\Big)\geqslant 2-\mathrm{mult}_{P}\big(D\big),%
$$
which implies that $\mathrm{mult}_{P}(D)\leqslant 1$. But
$\mathrm{mult}_{P}(D)\geqslant 1/\lambda>1$. Thus, we see that
$Q\not\in\bar{T}$.

As $T$ is irreducible, the surface $\bar{X}$ is a smooth quartic
hypersurface in $\mathbb{P}(1,1,1,2)$.

Let $\bar{M}$ be a general curve in $|-K_{\bar{X}}|$ such that
$Q\in\bar{M}$. Then
$$
3-\mathrm{mult}_{P}\big(D\big)=\bar{M}\cdot\bar{D}\geqslant\mathrm{mult}_{Q}\big(\bar{M}\big)\mathrm{mult}_{Q}\big(\bar{D}\big)\geqslant\mathrm{mult}_{Q}\big(\bar{D}\big),
$$
as $\bar{M}\not\subset\mathrm{Supp}(D)$.

Let  $\rho\colon\tilde{X}\to \bar{X}$ be a blow up of the~point
$Q$ and let $F$ be the~$\rho$-exceptional curve.~Then
$$
K_{\tilde{X}}+\lambda\tilde{D}+\Big(\lambda \mathrm{mult}_{P}\big(D\big)-1\Big)\tilde{E}+\Big(\lambda\mathrm{mult}_{Q}\big(\bar{D}\big)+\lambda\mathrm{mult}_{P}\big(D\big)-2\Big)F\sim_{\mathbb{Q}}\big(\sigma\circ\rho\big)^{*}\Big(K_{X}+\lambda D\Big),%
$$
where $\tilde{D}$ and $\tilde{E}_{i}$ are proper transforms of
the~divisors $D$ and $E$ on the~surface $\tilde{X}$, respectively.

It follows from Remark~\ref{remark:blow-up-inequality} that there
is a~point $O\in F$ such that
$$
\mathrm{LCS}\Bigg(\tilde{X}, \lambda\tilde{D}+\Big(\lambda
\mathrm{mult}_{P}\big(D\big)-1\Big)\tilde{E}+\Big(\lambda\mathrm{mult}_{Q}\big(\bar{D}\big)+\lambda\mathrm{mult}_{P}\big(D\big)-2\Big)F\Bigg)=O,%
$$
since
$\lambda\mathrm{mult}_{Q}(\bar{D})+\lambda\mathrm{mult}_{P}(D)-2\leqslant
3\lambda-2<1$. By Remark~\ref{remark:blow-up-inequality}, we have
\begin{equation}
\label{equation:dP3-inequality-I}
\mathrm{mult}_{O}\big(\tilde{D}\big)+\mathrm{mult}_{Q}\big(\bar{D}\big)+\mathrm{mult}_{P}\big(D\big)\geqslant\frac{3}{\lambda}>3.%
\end{equation}

If $O=\tilde{E}\cap F$, then it follows from
Lemma~\ref{lemma:adjunction} that
$$
2\lambda\mathrm{mult}_{P}\big(D\big)-2=\Bigg(\lambda\tilde{D}+\Big(\lambda\mathrm{mult}_{Q}\big(\bar{D}\big)+\lambda\mathrm{mult}_{P}\big(D\big)-2\Big)F\Bigg)\cdot\tilde{E}>1,%
$$
which implies that $\mathrm{mult}_{P}(D)>3/2$. However
$\mathrm{mult}_{P}(D)\leqslant 3/2$. Thus, we see that
$O\not\in\tilde{E}$.

There exists a unique curve $\tilde{B}$ in the pencil
$|-K_{\tilde{X}}|$ such that $O\in\tilde{B}$. Then
$$
\tilde{E}\not\subset\mathrm{Supp}\big(\tilde{B}\big)\not\supset F,
$$
since both $O\not\in\tilde{E}$ and $Q\not\in\bar{T}$. Put
$B=\sigma\circ\rho(\tilde{B})$. Then $B\in|-K_{X}|$ and $B\ne T$.

The curve $B$ is $G$-invariant, which implies that $B$ is
irreducible since $P\in B$.

By Remark~\ref{remark:convexity}, we may assume that
$B\not\subset\mathrm{Supp}(D)$. Then
$$
3-\mathrm{mult}_{P}\big(D\big)-\mathrm{mult}_{Q}\big(\bar{D}\big)=\tilde{B}\cdot\tilde{D}\geqslant\mathrm{mult}_{O}\big(\tilde{B}\big)\mathrm{mult}_{O}\big(\tilde{D}\big)\geqslant\mathrm{mult}_{O}\big(\tilde{D}\big),
$$
which is impossible by $(\ref{equation:dP3-inequality-I})$.
\end{proof}

Let us show how to compute $\mathrm{lct}(X,G)$  in one case.

\begin{lemma}\label{lemma:dP3-Z2-Z2-Z2}
Suppose that the surface $X$ is given by the equation
$$
x^3+x\Big(y^2+z^2+t^2\Big)+ayzt=0\subset\mathbb{P}^{3}\cong\mathrm{Proj}\Big(\mathbb{C}[x,y,z,t]\Big),%
$$
where $a$ is a general complex number. Then
$\mathrm{lct}(X,\mathrm{Aut}(X))=1$.
\end{lemma}

\begin{proof}
It follows from \cite{DoIs06} that
$$
\mathrm{Aut}\big(X\big)\cong\mathbb{S}_4,%
$$
which implies that the only $\mathrm{Aut}(X)$-invariant curve in
$|-K_{X}|$ is cut out on $X$ by $x=0$.

The only $\mathrm{Aut}(X)$-invariant curve in $|-K_{X}|$ has
ordinary double points, and $\mathrm{Aut}(X)$-invariant subgroup
in $\mathrm{Pic}(X)$ is generated by $-K_{X}$. Then
$\mathrm{lct}(X,\mathrm{Aut}(X))=1$ by
Theorem~\ref{theorem:dP3-lct-lct1}.
\end{proof}

\section{Intersection of two quadrics}
\label{sec:degree-four}

Let $X$ be a smooth complete intersection of two quadrics in
$\mathbb{P}^{4}$. Then $X$ can be given~by
$$
\sum_{i=0}^4 \alpha_i x_i^2=\sum_{i=0}^4\beta_i x_i^2=0\subset\mathbb{P}^{4}\cong\mathrm{Proj}\Big(\mathbb{C}[x_{0},x_{1},x_{2},x_{3},x_{4}]\Big)%
$$
for some
$[\alpha_{0}:\alpha_{1}:\alpha_{2}:\alpha_{3}:\alpha_{4}]\ne
[\beta_{0}:\beta_{1}:\beta_{2}:\beta_{3}:\beta_{4}]$ in
$\mathbb{P}^{4}$ (see \cite[Lemma~6.5]{DoIs06}).

The group $\mathrm{Aut}(X)$ is finite. Let $\tau_{1}$, $\tau_{2}$,
$\tau_{3}$, $\tau_{4}$ be involutions in $\mathrm{Aut}(X)$ such
that
$$
\left\{%
\aligned
&\tau_{1}\big([x_{0}:x_{1}:x_{2}:x_{3}:x_{4}]\big)=[x_{0}:-x_{1}:x_{2}:x_{3}:x_{4}],\\%
&\tau_{2}\big([x_{0}:x_{1}:x_{2}:x_{3}:x_{4}]\big)=[x_{0}:x_{1}:-x_{2}:x_{3}:x_{4}],\\%
&\tau_{3}\big([x_{0}:x_{1}:x_{2}:x_{3}:x_{4}]\big)=[x_{0}:x_{1}:x_{2}:-x_{3}:x_{4}],\\%
&\tau_{4}\big([x_{0}:x_{1}:x_{2}:x_{3}:x_{4}]\big)=[x_{0}:x_{1}:x_{2}:x_{3}:-x_{4}],\\%
\endaligned\right.%
$$
and let $\Gamma$ be a subgroup in $\mathrm{Aut}(X)$ that is
generated by $\tau_{1}$, $\tau_{2}$, $\tau_{3}$, $\tau_{4}$. Then
$\Gamma\cong\mathbb{Z}_{2}^{4}$.

\begin{lemma}[{\cite[Theorem~6.9]{DoIs06}}]
\label{lemma:dP4-Gamma} A $\Gamma$-invariant subgroup in
$\mathrm{Pic}(X)$ is generated by $-K_{X}$.
\end{lemma}

The surface $X$ contains no $\Gamma$-fixed points, which implies
the following result by
Corollary~\ref{corollary:weakly-exceptional}.

\begin{corollary}[{\cite[Example~1.10]{Ch07b}}]
\label{corollary:dP4-Gamma} The equality
$\mathrm{lct}(X,\Gamma)=1$ holds.
\end{corollary}

It easily follows from \cite{DoIs06} that the following two
conditions are equivalent:
\begin{itemize}
\item the linear system $|-K_X|$ does not contain~$\mathrm{Aut}(X)$-invariant curves,%
\item either $\mathrm{Aut}(X)\cong\mathbb{Z}_2^4\rtimes\mathbb{S}_3$ or $\mathrm{Aut}(X)\cong\mathbb{Z}_2^4\rtimes\mathbb{D}_{5}$.%
\end{itemize}

\begin{corollary}
\label{corollary:dP4-K} If
$\mathrm{Aut}(X)\not\cong\mathbb{Z}_2^4\rtimes\mathbb{S}_3$ and
$\mathrm{Aut}(X)\not\cong\mathbb{Z}_2^4\rtimes\mathbb{D}_{5}$,
then $\mathrm{lct}(X,\mathrm{Aut}(X))=1$.
\end{corollary}

The main purpose of this section is to prove the following result.

\begin{theorem}
\label{theorem:dP4-2K} If
$\mathrm{Aut}(X)\cong\mathbb{Z}_2^4\rtimes\mathbb{S}_3$ or
$\mathrm{Aut}(X)\cong\mathbb{Z}_2^4\rtimes\mathbb{D}_{5}$, then
$\mathrm{lct}(X,\mathrm{Aut}(X))=2$.
\end{theorem}

\begin{proof}
Suppose that either
$\mathrm{Aut}(X)\cong\mathbb{Z}_2^4\rtimes\mathbb{S}_3$ or
$\mathrm{Aut}(X)\cong\mathbb{Z}_2^4\rtimes\mathbb{D}_{5}$. Then
$$
\mathrm{lct}\big(X,G\big)\leqslant\mathrm{lct}_{2}\big(X,G\big)\leqslant 2,%
$$
since the linear system $|-2K_X|$ contains
a~$\mathrm{Aut}(X)$-invariant curve (see \cite{DoIs06}).

Suppose that $\mathrm{lct}(X,G)<2$. Let us derive a contradiction.

There exists a~$G$-invariant effective $\mathbb{Q}$-divisor $D$ on
the~surface $X$ such that
$$
D\sim_{\mathbb{Q}} -K_{X}
$$
and  $(X,\lambda D)$ is strictly log canonical for some
$\lambda\in\mathbb{Q}$ such that $\lambda<2$.

It follows from Lemmata~\ref{lemma:LCS-zero-dimensional} and
\ref{lemma:Nadel-IILC} that $|\mathrm{LCS}(X,\lambda
D)|\in\{2,3,5\}$ and $\mathrm{LCS}(X,\lambda D)$ imposes
independent linear conditions on hyperplanes in $\mathbb{P}^{4}$,
since $|-K_{X}|$ contains no $G$-invariant curves.

Suppose that
$\mathrm{Aut}(X)\cong\mathbb{Z}_2^4\rtimes\mathbb{S}_3$. Then
$|\mathrm{LCS}(X,\lambda D)|\ne 5$, and $X$ can be given by
$$
x_0^2+\epsilon_3 x_1^2+\epsilon_3^2 x_2^2+x_3^2=x_0^2+\epsilon_3^2 x_1^2+\epsilon_3 x_2^2+x_4^2=0\subset\mathbb{P}^{4}\cong\mathrm{Proj}\Big(\mathbb{C}[x_{0},x_{1},x_{2},x_{3},x_{4}]\Big),%
$$
where $\epsilon_{3}$ is a primitive cube root~of~unity. Let
$\iota_{1}$ and $\iota_{2}$ be elements in $\mathrm{Aut}(X)$ such
that
$$
\left\{%
\aligned
&\iota_{1}\big([x_0:x_1:x_2:x_3:x_4:x_5]\big)=[x_0:x_2:x_1:x_4:x_3],\\%
&\iota_{2}([x_0:x_1:x_2:x_3:x_4:x_5])=[x_1:x_2:x_0:\epsilon_3x_3:\epsilon_3^2x_4],\\%
\endaligned\right.%
$$
and let $\Pi$ be a linear subspace in $\mathbb{P}^{4}$ spanned by
$\mathrm{LCS}(X,\lambda D)$. Then
$$
\mathrm{Aut}\big(X\big)=\big\langle \Gamma,\iota_{1},\iota_{2}\big\rangle%
$$
furthermore, either we have $|\mathrm{LCS}(X,\lambda D)|=2$ and
$\Pi$ is given by the equations $x_{0}=x_{1}=x_{2}=0$, or we have
$|\mathrm{LCS}(X,\lambda D)|=3$ and $\Pi$ is given by
$x_{3}=x_{4}=0$. Since~the~subset
$$
x_0^2+\epsilon_3 x_1^2+\epsilon_3^2 x_2^2+x_3^2=x_0^2+\epsilon_3^2 x_1^2+\epsilon_3 x_2^2+x_4^2=x_{0}=x_{1}=x_{2}=0%
$$
is empty, we have $|\mathrm{LCS}(X,\lambda D)|=3$ and $\Pi$ is
given by $x_{3}=x_{4}=0$. However the subset
$$
x_0^2+\epsilon_3 x_1^2+\epsilon_3^2 x_2^2+x_3^2=x_0^2+\epsilon_3^2 x_1^2+\epsilon_3 x_2^2+x_4^2=x_{3}=x_{4}=0%
$$
consists of four points, which implies that
$|\mathrm{LCS}(X,\lambda D)|\ne 3$. Thus, we have
$\mathrm{Aut}(X)\not\cong\mathbb{Z}_2^4\rtimes\mathbb{S}_3$.

We see that
$\mathrm{Aut}(X)\cong\mathbb{Z}_2^4\rtimes\mathbb{D}_{5}$. Then
$|\mathrm{LCS}(X,\lambda D)|\ne 3$, and $X$ can be given by
$$
\sum_{i=0}^4 \epsilon_5^i x_i^2=\sum_{i=0}^4\epsilon_5^{4-i}x_i^2=0\subset\mathbb{P}^{4}\cong\mathrm{Proj}\Big(\mathbb{C}[x_{0},x_{1},x_{2},x_{3},x_{4}]\Big)%
$$
where $\epsilon_{5}$ is a primitive fifth root of unity. Let
$\chi_{1}$ and $\chi_{2}$ be elements in $\mathrm{Aut}(X)$ such
that
$$
\left\{%
\aligned
&\chi_{1}\big([x_0:x_1:x_2:x_3:x_4:x_5]\big)=[x_1:x_2:x_3:x_4:x_0],\\%
&\chi_{2}([x_0:x_1:x_2:x_3:x_4:x_5])=[x_4:x_3:x_2:x_1:x_0],\\%
\endaligned\right.%
$$
and let $\Pi$ be a linear subspace in $\mathbb{P}^{4}$ spanned by
$\mathrm{LCS}(X,\lambda D)$. Then
$$
\mathrm{Aut}\big(X\big)=\big\langle\Gamma,\chi_{1},\chi_{2}\big\rangle%
$$
and $\Pi\not\cong\mathbb{P}^{1}$. Since $|\mathrm{LCS}(X,\lambda
D)|\in\{2,5\}$, we have $|\mathrm{LCS}(X,\lambda D)|=5$, which is
impossible because the surface $X$ does not have
$\mathrm{Aut}(X)$-orbits of length $5$.
\end{proof}

\begin{corollary}
\label{corollary:dP4-2K} The following four conditions are
equivalent:
\begin{itemize}
\item the linear system $|-K_X|$ does not contain~$\mathrm{Aut}(X)$-invariant curves,%
\item either $\mathrm{Aut}(X)\cong\mathbb{Z}_2^4\rtimes\mathbb{S}_3$ or $\mathrm{Aut}(X)\cong\mathbb{Z}_2^4\rtimes\mathbb{D}_{5}$,%
\item the inequality $\mathrm{lct}(X,\mathrm{Aut}(X))>1$ holds,
\item the equality $\mathrm{lct}(X,\mathrm{Aut}(X))=2$ holds.
\end{itemize}
\end{corollary}

\section{Surfaces of big degree}
\label{sec:big-degree}

Let $X$ be a smooth del Pezzo surface and let $G$ be a finite
subgroup in $\mathrm{Aut}(X)$.

\begin{lemma}
\label{lemma:dP6} Suppose that $K_{X}^{2}=6$. Then
$\mathrm{lct}(X,G)\leqslant 1$.
\end{lemma}

\begin{proof}
Let $L_{1}$, $L_{2}$, $L_{3}$, $L_{4}$, $L_{5}$ and $L_{6}$ be
smooth rational curves on the surface $X$ such that
$$
L_{1}\cdot L_{1}=L_{2}\cdot L_{2}=L_{3}\cdot L_{3}=L_{4}\cdot L_{4}=L_{5}\cdot L_{5}=L_{6}\cdot L_{6}=-1%
$$
and $L_{i}\ne L_{j}\iff i\ne j$. Then $\sum_{i=1}^{6}L_{i}$ is a
$G$-invariant curve in $|-K_{X}|$.
\end{proof}

\begin{lemma}
\label{lemma:dP7} Suppose that $K_{X}^{2}=7$. Then
$\mathrm{lct}(X,G)=1/3$.
\end{lemma}

\begin{proof}
Let $L_{1}$, $L_{2}$ and $L_{3}$ be smooth rational curves on the
surface $X$ such that
$$
L_{1}\cdot L_{1}=L_{2}\cdot L_{2}=L_{3}\cdot L_{3}=-L_{1}\cdot L_{2}=-L_{3}\cdot L_{2}=-1%
$$
and $L_{1}\cdot L_{2}=0$. Then $2L_{1}+3L_{2}+L_{1}\in |-K_{X}|$
and the curve $2L_{1}+3L_{2}+L_{1}$ is $G$-invariant, which
immediately implies that $\mathrm{lct}(X,G)=1/3$ by
Example~\ref{example:GAFA}.
\end{proof}

\begin{lemma}
\label{lemma:dP8} Suppose that $K_{X}^{2}=8$ and
$X\not\cong\mathbb{P}^{1}\times\mathbb{P}^{1}$. Then
$\mathrm{lct}(X,G)\leqslant 1/2$.
\end{lemma}

\begin{proof}
Let $L$ and $E$ be smooth rational curves on the surface $X$ such
that $L\cdot L=0$~and~$E\cdot E=-1$, and let $C$ be a
$G$-invariant curve in the linear system $|nL|$ for some $n\gg 0$.
Then
$$
2E+\frac{3}{n}C\sim_{\mathbb{Q}} -K_{X},
$$
which implies that $\mathrm{lct}(X,G)\leqslant 1/2$, since $E$ is
$G$-invariant.
\end{proof}

\begin{corollary}
\label{corollary:big-degree} If $\mathrm{lct}(X,G)>1$ and
$K_{X}^{2}\geqslant 6$, then either $X\cong\mathbb{P}^{2}$ or
$X\cong\mathbb{P}^{1}\times\mathbb{P}^{1}$.
\end{corollary}

Let us conclude this section by proving the following criterion
(cf. Example~\ref{example:complete-answer}).

\begin{theorem}\label{theorem:smooth-quadric}
Suppose that $X\cong\mathbb{P}^{1}\times\mathbb{P}^{1}$. Then the
following are equivalent:
\begin{itemize}
\item the inequality $\mathrm{lct}(X,G)>1$ holds, %
\item the inequality $\mathrm{lct}(X,G)\geqslant 5/4$ holds,%

\item there are no $G$-invariant curves in the linear systems
$$
\big|L_1\big|, \big|L_2\big|, \big|2L_1\big|, \big|2L_2\big|, \big|L_1+L_2\big|, \big|L_1+2L_2\big|, \big|2L_1+L_2\big|, \big|2L_1+2L_2\big|,%
$$
where $L_{1}$ and $L_{2}$ are fibers of two distinct
natural~projections of the~surface $X$ to $\mathbb{P}^{1}$.
\end{itemize}
\end{theorem}

\begin{proof}
Let $L_{1}$ and $L_{2}$ be fibers of two distinct
natural~projections of the~surface $X$ to $\mathbb{P}^{1}$. Then
$$
\big|aL_1+bL_2\big|
$$
contains no $G$-invariant curves for every $a$ and $b$ in
$\{0,1,2\}$ whenever $\mathrm{lct}(X,G)>1$.

Suppose that $|L_1|$, $|L_2|$, $|2L_1|$, $|2L_2|$, $|L_1+L_2|$,
$|L_1+2L_2|$, $|2L_1+L_2|$, $|2L_1+2L_2|$ do not contain
$G$-invariant~curves and $\mathrm{lct}(X,G)<4/3$. Let us derive a
contradiction.

There exists a~$G$-invariant effective $\mathbb{Q}$-divisor $D$ on
the~surface $X$ such that
$$
D\sim_{\mathbb{Q}} 2\big(L_1 + L_2\big) \sim -K_X
$$
and $\big(X,\lambda D\big)$ is strictly log canonical for some
$\lambda\in\mathbb{Q}$ such that $\lambda<5/4$. By
Theorem~\ref{theorem:Shokurov-vanishing},~we~have
$$
H^1\Big(X, \mathcal{I}\big(X,\lambda D\big) \otimes \mathcal{O}_{X}\big(L_1+L_2\big)\Big)=0,%
$$
where $\mathcal{I}(X,\lambda D)$ is the~multiplier ideal sheaf of
the log pair $(X,\lambda D)$
(see~Section~\ref{section:preliminaries}).

The ideal sheaf $\mathcal{I}(X,\lambda D)$  defines
a~zero-dimensional subscheme $\mathcal{L}$ of the~surface $X$,
since the~linear system $|aL_1+bL_2|$ has no $G$-invariant
curves for every $a$ and $b$ in $\{0,1,2\}$.

Since the~subscheme $\mathcal{L}$ is zero-dimensional, we have
the~short exact sequence
$$
0\longrightarrow H^0\Big(X, \mathcal{I}\big(X,\lambda D\big)\otimes \mathcal{O}_{X}\big(L_1+L_2\big)\Big)\longrightarrow H^0\Big(X, \mathcal{O}_{X}\big(L_1+L_2\big)\Big)\longrightarrow H^0\big(\mathcal{O}_{\mathcal{L}}\big)\longrightarrow 0,%
$$
which implies that  $\mathrm{Supp}(\mathcal{L})$ consists of four
points that are not contained in one curve in $|L_1+L_2|$.

Let $P_{1}$, $P_{2}$, $P_{3}$ and $P_{4}$ be four points in
$\mathrm{Supp}(\mathcal{L})$. Then  $P_{1}$, $P_{2}$, $P_{3}$ and
$P_{4}$ form a~$G$-orbit.

Write $L_{11}$, $L_{12}$, $L_{13}$, $L_{14}$ for the curves in
$|L_{1}|$ that pass through $P_{1}$, $P_{2}$, $P_{3}$, $P_{4}$,
respectively, write~$L_{21}$, $L_{22}$, $L_{23}$, $L_{24}$ for the
curves in $|L_{2}|$ that pass through $P_{1}$, $P_{2}$, $P_{3}$,
$P_{4}$, respectively.~Then
$$
L_{1i}=L_{1j}\iff i=j\iff L_{2i}=L_{2j},
$$
as $|L_{1}|$, $|L_{2}|$ and $|L_1+L_2|$ do not contain
$G$-invariant curves.

Let $C_{1}$, $C_{2}$, $C_{3}$, $C_{4}$ be the curves in the linear
system $|L_{1}+L_{2}|$ such that each contains exactly
three points in $\mathrm{Supp}(\mathcal{L})$ and $P_{1}\not\in
C_{1}$, $P_{2}\not\in C_{2}$,  $P_{3}\not\in C_{3}$, $P_{4}\not\in
C_{4}$. Then
$$
\Bigg(X, \frac{2}{3}\Big(C_{1}+C_{2}+C_{3}+C_{4}\Big)\Bigg)
$$
is strictly log canonical, since the curves $C_{1}$, $C_{2}$,
$C_{3}$, $C_{4}$ are smooth and irreducible.

By Remark~\ref{remark:convexity}, we may assume that
$\mathrm{Supp}(D)$ does not contain $C_{1}$, $C_{2}$, $C_{3}$ and
$C_{4}$.~Then
$$
16=D\cdot\Big(C_{1}+C_{2}+C_{3}+C_{4}\Big)=3\sum_{i=1}^{4}\mathrm{mult}_{P_{i}}\big(D\big)=12\mathrm{mult}_{P_{1}}\big(D\big)=\cdots=12\mathrm{mult}_{P_{4}}\big(D\big),
$$
which implies that
$\mathrm{mult}_{P_{1}}(D)=\mathrm{mult}_{P_{2}}(D)=\mathrm{mult}_{P_{3}}(D)=\mathrm{mult}_{P_{4}}(D)\leqslant
4/3$.

Let $\sigma\colon\bar{X}\to X$ be the~blow-up of the~points $P_1$,
$P_2$, $P_{3}$ and $P_{4}$, let $E_{1}$, $E_{2}$, $E_{3}$ and
$E_{4}$ be the~$\sigma$-exceptional curves such that
$\sigma(E_{1})=P_{1}$, $\sigma(E_{2})=P_{2}$,
$\sigma(E_{3})=P_{3}$ and $\sigma(E_{4})=P_{4}$.~Then
$$
K_{\bar{X}}+\lambda\bar{D}+\sum_{i=1}^{4}\Big(\lambda\mathrm{mult}_{P_i}\big(D\big)-1\Big)E_i\sim_{\mathbb{Q}}\sigma^{*}\Big(K_X + \lambda D\Big),%
$$
where $\bar{D}$ is the~proper transform of the~divisor $D$ on
the~surface $\bar{X}$.

By Remark~\ref{remark:blow-up-inequality}, there are points
$Q_{1}\in E_{1}$, $Q_{2}\in E_{2}$, $Q_{3}\in E_{3}$ and $Q_{4}\in
E_{4}$ such that
$$
\mathrm{LCS}\Bigg(\bar{X},\lambda\bar{D}+\sum_{i=1}^{4}\Big(\lambda\mathrm{mult}_{P_i}\big(D\big)-1\Big)E_i\Bigg)=\Big\{Q_{1},Q_{2},Q_{3},Q_{4}\Big\},%
$$
since
$\lambda\mathrm{mult}_{P_{1}}(D)=\mathrm{mult}_{P_{2}}(D)=\mathrm{mult}_{P_{3}}(D)=\mathrm{mult}_{P_{4}}(D)\leqslant
5/3<2$.

Since $\bar{D}$ is $G$-invariant, it follows that the~action of
the~group $G$ on the~surface $X$ naturally lifts to an action on
$\bar{X}$ where the~points $Q_{1}$, $Q_{2}$, $Q_{3}$ and $Q_{4}$
form a~$G$-orbit.

Put $\bar{R}=3\sigma^{*}(L_{1}+L_{2})-2\sum_{i=1}^{4}E_{i}$. Then
$\bar{R}\cdot\bar{R}=4$, which implies that $\bar{R}$ is nef and
big, since
$$
\bar{L}_{11}+\bar{L}_{21}+2\bar{C}_{1}\sim 3\sigma^{*}\Big(L_{1}+L_{2}\Big)-2\sum_{i=1}^{4}E_{i}%
$$
and $\bar{L}_{11}\cdot\bar{R}=\bar{L}_{21}\cdot\bar{R}=1$ and
$\bar{C}_{1}\cdot\bar{R}=0$, where we denote by symbols
$\bar{L}_{11}$, $\bar{L}_{21}$ and $\bar{C}_{1}$ the~proper
transforms of the~curves $L_{11}$, $L_{21}$ and $C_{1}$ on the
surface $\bar{X}$, respectively. Then
$$
K_{\bar{X}}+\lambda\bar{D}+\sum_{i=1}^{4}\Big(\lambda\mathrm{mult}_{P_i}\big(D\big)-1\Big)E_i+\frac{1}{2}\Bigg(\bar{R}+\Big(5-4\lambda\Big)\sigma^{*}\Big(L_{1}+L_{2}\Big)\Bigg)\sim_{\mathbb{Q}}2\sigma^{*}\Big(L_{1}+L_{2}\Big)-\sum_{i=1}^{4}E_{i}\sim -K_{\bar{X}},%
$$
where $\bar{R}+(5-4\lambda)\sigma^{*}(L_{1}+L_{2})$ is nef and
big since $\lambda<5/4$. By
Theorem~\ref{theorem:Shokurov-vanishing},~we~have
$$
H^1\Bigg(X, \mathcal{I}\Big(\bar{X},\lambda\bar{D}+\sum_{i=1}^{4}\Big(\lambda\mathrm{mult}_{P_i}\big(D\big)-1\Big)E_i\Big)\otimes \mathcal{O}_{\bar{X}}\Big(-K_{\bar{X}}\Big)\Bigg)=0,%
$$
from which it follows that there is a unique curve $\bar{C}\in
|-K_{\bar{X}}|$ containing $Q_{1}$, $Q_{2}$, $Q_{3}$ and $Q_{4}$.

The curve $\bar{C}$ must be $G$-invariant, however then
$\sigma(\bar{C})$ is also $G$-invariant, which is impossible,
since $\sigma(\bar{C})\in|2L_{1}+2L_{2}|$ and $|2L_{1}+2L_{2}|$
contains no $G$-invariant curves.
\end{proof}

\end{document}